%% file: main.tex
\documentclass[11pt]{amsart}
\usepackage[colorlinks, linkcolor=blue, citecolor=blue, pagebackref]{hyperref}
\usepackage{amsrefs}
\usepackage{amssymb}
\usepackage{ytableau}
\usepackage{xcolor}
\usepackage{mathtools}
\usepackage{tabularray}
\usepackage{subcaption}
\usepackage{fullpage}
\usepackage[indent]{parskip}
\usepackage{bm}
\usepackage{enumitem}
\usepackage{setspace}
\usepackage{float}

\theoremstyle{plain}
\newtheorem{thm}{Theorem}[section]

\newtheorem{lemma}[thm]{Lemma}
\newtheorem{cor}[thm]{Corollary}

\newtheorem*{thm*}{Theorem}

\theoremstyle{definition}
\newtheorem{dfn}[thm]{Definition}
\newtheorem{ex}[thm]{Example}
\newtheorem{rem}[thm]{Remark}

\numberwithin{equation}{section}

\newcommand{\C}{\mathbb{C}}
\newcommand{\R}{\mathbb{R}}
\newcommand{\GL}{{\rm GL}}
\newcommand{\gl}{\mathfrak{gl}}
\newcommand{\so}{\mathfrak{so}}
\newcommand{\g}{\mathfrak{g}}
\newcommand{\M}{{\rm M}}
\newcommand{\LRC}[2]{c^{#1}_{#2}}

\newcommand{\bmu}{\bm{\mu}}
\newcommand{\la}{\lambda}
\newcommand{\ga}{\gamma}
\newcommand{\bnu}{\bm{\nu}}
\renewcommand{\hom}{{\rm Hom}}
\renewcommand{\O}{{\rm O}}
\newcommand{\Sp}{{\rm Sp}}
\renewcommand{\sp}{\mathfrak{sp}}
\renewcommand{\k}{\mathfrak{k}}
\newcommand{\p}{\mathfrak{p}}
\renewcommand{\q}{\mathfrak{q}}
\newcommand{\SSYT}{{\rm SSYT}}

\allowdisplaybreaks

\begin{document}

\title{$K$-type multiplicities in degenerate principal series \\ via Howe duality}

\author{Mark Colarusso}
\address{
Department of Mathematics and Statistics \\
University of South Alabama \\ 
Mobile, AL 36608} 
\email{mcolarusso@southalabama.edu}

\author{William Q.~Erickson}
\address{
Department of Mathematics\\
Baylor University \\ 
One Bear Place \#97328\\
Waco, TX 76798} 
\email{will\_erickson@baylor.edu}

\author{Andrew Frohmader}
\address{
Department of Mathematical Sciences\\
University of Wisconsin--Milwaukee \\ 
3200 N.~Cramer St.\\
Milwaukee, WI 53211} 
\email{frohmad4@uwm.edu}

\author{Jeb F.~Willenbring}
\address{
Department of Mathematical Sciences\\
University of Wisconsin--Milwaukee \\ 
3200 N.~Cramer St.\\
Milwaukee, WI 53211} 
\email{jw@uwm.edu}



\begin{abstract}

Let $K$ be one of the complex classical groups $\O_k$, $\GL_k$, or $\Sp_{2k}$.
Let $M \subseteq K$ be the block diagonal embedding $\O_{k_1} \times \cdots \times \O_{k_r}$ or $\GL_{k_1} \times \cdots \times \GL_{k_r}$ or $\Sp_{2k_1} \times \cdots \times \Sp_{2k_r}$, respectively.
By using Howe duality and seesaw reciprocity as a unified conceptual framework, we prove a formula for the branching multiplicities from $K$ to $M$ which is expressed as a sum of generalized Littlewood--Richardson coefficients, valid within a certain stable range.
By viewing $K$ as the complexification of the maximal compact subgroup $K_\R$ of the real group $G_\R = \GL(k,\R)$, $\GL(k, \C)$, or $\GL(k,\mathbb{H})$, respectively, one can interpret our branching multiplicities as $K_\R$-type multiplicities in degenerate principal series representations of $G_\R$.
Upon specializing to the minimal $M$, where $k_1 = \cdots = k_r = 1$, we establish a fully general tableau-theoretic interpretation of the branching multiplicities, corresponding to the $K_\R$-type multiplicities in the principal series.
\end{abstract}

\maketitle

\section{Introduction}

\subsection*{Motivation from real groups}

The principal series plays a fundamental role in the representation theory of real groups.
Casselman~\cite{Casselman} established that every irreducible admissible representation of a real reductive group $G_\R$ can be realized as a submodule of some principal series representation of $G_\R$.
(See also~\cite{Collingwood}*{Cor.~1.18}.)
Principal series representations arise via induction from a finite-dimensional irreducible representation of a minimal parabolic subgroup $P_\R$.
More generally, if $P_\R$ is not necessarily minimal, then the parabolically induced representations are said to be \emph{degenerate principal series} representations.
These representations are infinite-dimensional.
A key technique in infinite-dimensional representation theory is to decompose admissible representations of $G_\R$ under restriction to the action of a maximal compact subgroup $K_\R$.
This allows for the possibility of a combinatorial description of the structure of the representation by determining a formula for the multiplicities of the irreducible representations of $K_\R$ (called \emph{$K_\R$-types}) in the decomposition.
Although the main result in this paper is a finite-dimensional branching rule from $K_\R$ down to a certain subgroup $M_\R$, this rule can be viewed as a $K_\R$-type multiplicity formula for degenerate principal series of $G_\R$.

In more detail, let $G_\R$ be a real reductive group with maximal compact subgroup $K_\R$.
Let $P_\R$ be a (not necessarily minimal) parabolic subgroup of $G_\R$.
Let $\bmu$ be an irreducible finite-dimensional representation of $P_\R$ on which its nilradical acts trivially.
By inducing $\bmu$ from $P_\R$ to $G_\R$, in the sense of~\cite{WallachRRI}*{\S1.5}, we obtain a \emph{degenerate principal series} representation of $G_\R$, whose underlying Harish-Chandra module we denote by $\mathcal{P}(\bmu)$:
\begin{equation*}
    \label{Ind P G}
    \mathcal{P}(\bmu) \coloneqq \text{the underlying Harish-Chandra module of }{\rm Ind}^{G_\R}_{P_\R} \bmu.
\end{equation*}

From our perspective, the goal is to understand the $K_\R$-structure of $\mathcal{P}(\bmu)$.
To this end, let $M_\R$ denote the intersection of $K_\R$ with the reductive Levi factor of $P_\R$.
In the context of this paper, we may suppose that $\bmu$ remains irreducible upon restriction from $P_\R$ to $M_\R$.
As representations of $K_\R$,
\begin{equation}
\label{Res Ind}
    \mathcal{P}(\bmu) \cong {\rm Ind}^{K_\R}_{M_\R} \bmu,
\end{equation}
where the right-hand side denotes algebraic induction from $M_\R$ to $K_\R$.
Now we wish to describe the multiplicity of a given $K_\R$-type, say $\lambda$, inside $\mathcal{P}(\bmu)$.
By applying Frobenius reciprocity to~\eqref{Res Ind}, we compute this multiplicity as
\begin{align*}
\dim \hom_{K_\R} ( \la, \mathcal{P}(\bmu) ) &= \dim \hom_{K_\R} ( \la, {\rm Ind}^{K_\R}_{M_\R} \bmu) \\
& = \dim \hom_{M_\R} ( \bmu , {\rm Res}^{K_\R}_{M_\R} \la ) \eqqcolon b^\la_{\bmu},
\end{align*}
where $b^\la_{\bmu}$ denotes the multiplicity of the $M_\R$-type $\bmu$ inside the $K_\R$-type $\la$.
In this way, the problem of determining $K_\R$-type multiplicities in degenerate principal series is equivalent to the restriction problem (for \emph{finite}-dimensional representations) from $K_\R$ to $M_\R$.
In the case where $P_\R$ is a minimal parabolic, the numbers $b^\la_{\bmu}$ are the $K_\R$-type multiplicities in the \emph{principal series} representation induced from $\bmu$.
The branching multiplicities $b^\la_{\bmu}$ are the focus of this paper.

\subsection*{Overview of results}

Recall that the finite-dimensional real associative division algebras are the real numbers $\R$, the complex numbers $\C$, and the quaternions $\mathbb{H}$.
In this paper, we let $G_\R$ be one of the real groups $\GL(k, \R)$, $\GL(k, \mathbb{C})$, or $\GL(k, \mathbb{H})$.
We choose $K_\R$ to be the subgroup of $G_\R$ preserving the norm induced by conjugation on $\R$, $\C$, or $\mathbb{H}$, respectively.
We take the subgroup $M_\R \subseteq K_\R$ to be the block-diagonally embedded direct sum of lower-rank groups of the same Lie type as $K_\R$; thus each possible $M_\R$ is given by a choice of positive integers $k_1, \ldots, k_r$ summing to $k$.
We let $K$ and $M$ denote the complexifications of $K_\R$ and $M_\R$.
Concrete details are given below:
\begingroup
\renewcommand*{\arraystretch}{1.5}
\begin{equation}
\label{G info}
    \begin{array}{c|c|c|c|c}
        G_\R & K_\R & M_\R & K & M \\ \hline
       \GL(k, \R) & \O(k) & \O(k_1) \times \cdots \times \O(k_r) & \O_k & \O_{k_1} \times \cdots \times \O_{k_r} \\
        \GL(k, \mathbb{C}) & {\rm U}(k) & {\rm U}(k_1) \times \cdots \times {\rm U}(k_r) & \GL_k & \GL_{k_1} \times \cdots \times \GL_{k_r} \\
        \GL(k, \mathbb{H}) & \Sp(k) & \Sp(k_1) \times \cdots \times \Sp(k_r) & \Sp_{2k} & \Sp_{2k_1} \times \cdots \times \Sp_{2k_r}
    \end{array}
\end{equation}
\endgroup
Note that the groups $K$ are the complex classical groups (orthogonal, general linear, and symplectic). 
For each of these groups $K$, the finite-dimensional irreducible complex representations $U^\la$ (with rational matrix coefficients) are indexed by a certain set of partitions $\la$, or pairs $\la = (\la^+, \la^-)$ of partitions for $K = \GL_k$.
For $M$, then, the finite-dimensional irreducible rational representations $W^{\bmu}$ are labeled by $r$-tuples $\bmu = (\mu_1, \ldots, \mu_r)$ of partitions, or (when $K = \GL_k$) by pairs $\bmu = (\bmu^+, \bmu^-)$ of such $r$-tuples.
As described above, our main interest in this paper is the restriction problem from $K_\R$ to $M_\R$, or equivalently from $K$ to $M$.
That is, our main results are combinatorial formulas for the branching multiplicities
\begin{equation}
\label{b la mu original}
    b^\la_{\bmu} \coloneqq \dim \hom_{M} ( W^{\bmu}, {\rm Res}^{K}_{M} U^\la ).
\end{equation}

The \emph{Littlewood--Richardson coefficients} are ubiquitous in algebraic combinatorics~\cites{Fulton,HoweLee}; given partitions $\la, \mu, \nu$, the Littlewood--Richardson coefficient $c^\la_{\mu \nu}$ gives the coefficient of the Schur function $s_\la$ in the product $s_\mu s_\nu$.
One can generalize these coefficients for products of arbitrarily many Schur functions, so that, for example, $c^\la_{\bmu \bm{\nu}}$ gives the coefficient of $s_\la$ in the product $s_{\mu_1} \cdots s_{\mu_r} s_{\nu_1} \cdots s_{\nu_t}$.
Recall that a partition is said to have \emph{even rows} (resp., columns) if every row (resp., column) in its Young diagram has even length. 
Our first main result is the following formula (see Theorem~\ref{thm:main result}) for the branching multiplicity $b^\la_{\bmu}$ expressed as a sum of generalized Littlewood--Richardson coefficients.

\begin{thm}[Stable branching rule from $K$ to $M$]
\label{theorem 1 in intro}

Let $b^\la_{\bmu}$ denote the branching multiplicity~\eqref{b la mu original} from $K$ to $M$.
Let $k_1, \ldots, k_r$ be positive integers summing to $k$.
If $\la$ and $\bmu$ lie within a certain stable range (to be specified below) depending on $\min_i \{k_i\}$, then we have the following:

\begin{enumerate}[label=\textup{(\alph*)}]
    \item Branching from $K = \O_k$ to $M = \O_{k_1} \times \cdots \times \O_{k_r}$, we have
    \[
    b^\la_{\bmu} = \sum_{\bnu} \: \LRC{\la}{\bmu\bnu},
    \]
where $\bnu$ ranges over all $(r-1)$-tuples of partitions with even rows.

\item Branching from $K = \GL_k$ to $M = \GL_{k_1} \times \cdots \times \GL_{k_r}$, we have 
\[
b^\la_{\bmu} =  \sum_{\bnu} \LRC{\la^+}{\bmu^+\bnu} \:\LRC{\la^-}{\bmu^-\bnu},
\]
where $\bnu$ ranges over all $(r-1)$-tuples of partitions.

\item Branching from $K = \Sp_{2k}$ to $M = \Sp_{2k_1} \times \cdots \times \Sp_{2k_r}$, we have 
\[
b^\la_{\bmu} = \sum_{\bnu} \LRC{\la}{\bmu\bnu},
\]
where $\bnu$ ranges over all $(r-1)$-tuples of partitions with even columns.

\end{enumerate}

\end{thm}

The specific ``stable range'' mentioned above can be found in the full statement of Theorem~\ref{thm:main result} below.
Essentially, the stable range is an upper bound on the number of parts allowed in the partition $\la$ and in each of the partitions in $\bmu$, and this upper bound depends on the smallest of the parameters $k_i$ given in~\eqref{G info}.
In other words, for fixed $\la$ and $\bmu$, the formulas in the theorem above are valid as long as all the factors in $M$ have sufficiently high rank.

As our second main result, we obtain a fully general and more concrete multiplicity formula, by restricting our attention to the special case where $M$ is the minimal direct sum embedding, that is, where  $k_1 = \cdots = k_r = 1$.
(This corresponds to the special case mentioned above where $P_\R$ is a minimal parabolic subgroup, meaning that our branching rule gives the $K_\R$-type multiplicities in the \emph{principal series}.)
In this special case where $M$ is minimal, we use the letter $\bm{\delta}$ rather than $\bmu$ to parametrize representations of $M$, to emphasize that $\bm{\delta}$ is a vector of integers (in particular, certain ``differences'' $\delta_i$) rather than of partitions.
Specifically, we realize the multiplicity $b^\la_{\bm{\delta}}$ as the number of tableaux in a certain set $\mathcal{T}(K)^\la_{\bm{\delta}}$.
These ``$K$-tableaux'' are defined in Definition~\ref{def:KC-tableaux} and in~\eqref{T(KC)}.
Roughly speaking, the parameter $k$ determines the maximum entry in these semistandard tableaux, the partition $\la$ determines the shape of the tableaux,  and the integer $k$-tuple $\bm{\delta}$ determines their weight, which encodes the frequencies of the entries.
Our Theorem~\ref{thm:K to M} takes the following form:

\begin{thm}[Branching rule from $K$ to the minimal $M$]
\label{theorem 2 in intro}

Let $K$ and $M$ be as given in~\eqref{G info}, where $M$ is given by the parameters $k_1 = \cdots = k_r = 1$.
Then we have
\[
b^\la_{\bm{\delta}} = \# \mathcal{T}(K)^\la_{\bm{\delta}},
\]
where $\mathcal{T}(K)^\la_{\bm{\delta}}$ is the set of $K$-tableaux with shape $\la$ and weight $\bm{\delta}$, to be defined in~\eqref{T(KC)}.

\end{thm}

In the corollaries in Section~\ref{sec:examples}, we record explicit formulas for $b^\la_{\bm{\delta}}$ in the special cases where $\la$ is given by one-row or one-column shapes.

\subsection*{Unified approach via Howe duality}

The contribution of this paper consists of not only the two main results highlighted above, but also the philosophy that unites all of the proofs.
In particular, we approach the problem via \emph{Howe duality}, whereby each of our three compact groups $K_\R$ is paired with a real noncompact group $G'_\R$.
(See details in Theorem~\ref{thm:Howe duality}.)
In each of the three Howe duality settings, we have a multiplicity-free decomposition of a certain space under the joint action of $K$ and the complexified Lie algebra $\g'$; in this decomposition, the $K$-modules are finite-dimensional, while the $\g'$-modules are infinite-dimensional.
By using \emph{seesaw reciprocity} in each Howe duality setting, we are able to obtain the desired (finite-dimensional) branching multiplicities from $K$ to $M$, by instead finding the (infinite-dimensional) branching multiplicities from $(\g')^{\oplus r}$ to $\g'$.
Although it may seem surprising that the infinite-dimensional branching rule is the easier one, we are able to exploit the well-known $\k'$-structure of the $\g'$-modules in question.
Despite this uniform approach, each classical group $K$ presents a different complication which we address in the proofs: $\O_k$ is disconnected, $\GL_k$ has rational representations which are not polynomial, and $\Sp_{2k}$ has a branching rule down to $\Sp_{2(k-1)}$ which is not multiplicity-free.

Through the lens of Howe duality, this paper also serves as a retrospective look at the following restriction problems: from $\O_k$ down to $\O_{k-1} \times \O_1$, from $\GL_k$ down to $\GL_{k-1} \times \GL_1$, and from $\Sp_{2k}$ down to $\Sp_{2(k-1)} \times \Sp_2$.
The corresponding branching rules were written down some time ago, in various degrees of explicitness, by King~\cite{King75}*{equations (4.8), (4.14--15)} and Proctor~\cite{Proctor}*{Props.~10.1 and 10.3}; see also~\cites{KY,KT,Lepowsky,Yacobi}.
In the course of proving Theorem~\ref{thm:K to M}, we obtain these branching rules in a uniform manner as a consequence of Howe duality.
In this respect, our paper is similar in spirit to~\cite{HTW}, which likewise used Howe duality as a unifying framework for (stable) branching rules for classical groups.
The present paper is also a natural sequel to our previous work~\cite{CEW}, where we used this approach to express tensor multiplicities for classical groups via contingency tables.
In general, similar restriction problems can also be approached (at least for connected groups) using the theory of crystal graphs~\cite{BumpSchilling} or Littelmann path algebras~\cites{Littelmann94,Littelmann95}.

We point out that our ``$K$-tableaux'' in Theorem~\ref{thm:K to M} are a synthesis of the orthogonal~\cite{KingEl}, rational~\cites{Stembridge,KingGYT}, and symplectic~\cite{ProctorRSK}*{p.~30} tableaux introduced by King, El Sharkaway, Stembridge, and Proctor.
Several authors subsequently created variations of orthogonal tableaux, such as those of King--Welsh~\cite{KingWelshOrthogonal}, Koike--Terada~\cite{KT}, Proctor~\cite{Proctor}, and Sundaram~\cite{Sundaram}.
Also closely related to our $K_\R$-type multiplicity formulas is the recent preprint~\cite{Frohmader} by the third author, which uses yet another set of orthogonal tableaux (from the literature on crystal bases) to give a combinatorial formula for graded multiplicity in the Kostant--Rallis setting for the symmetric pair $(\GL_n, \O_n)$.

\subsection*{Acknowledgments}

We would like to thank Nolan Wallach for his helpful comments that improved the clarity and precision of the introduction.

\section{Preliminaries}

\subsection*{Vectors of partitions}

Throughout the paper, we write $\mathbb{N}$ to denote the set of nonnegative integers.
A \emph{partition} is a weakly decreasing finite sequence of positive integers (\emph{parts}).
It is typical to identify a partition $\la$ with its \emph{Young diagram}, obtained by arranging boxes in left-justified rows such that the row lengths from top to bottom are given by the parts of $\la$.
In this way, one speaks of the ``rows'' or ``columns'' of $\la$ by viewing it as a Young diagram.
We write $\ell(\la)$ to denote the length of $\la$, which is also the number of rows in its Young diagram.
We say that $\la$ has \emph{even rows} (resp., \emph{columns}) if all rows (resp., columns) contain an even number of boxes.
We write $0$ to denote the empty partition.
We write $\mu \subseteq \la$ if the Young diagram of $\mu$ is contained in that of $\lambda$, where both Young diagrams are aligned with respect to their upper-left corners.

We use boldface Greek letters to denote vectors of partitions $\mu_i$:
\begin{equation}
\label{bold lambda}
    \bmu = (\mu_1, \ldots, \mu_r).
\end{equation}
We note that this use of subscripts is somewhat nonstandard, since in the literature one often sees $\mu_i$ denoting the $i$th component of a partition $\mu$.
In this paper, the notation $\mu_i$ never refers to the $i$th part of a partition $\mu$; indeed, we will have no need to reference individual parts at all.

In the case of the group $\GL_k$, we will need to generalize partitions by allowing nonpositive parts.
We will fix the total number of parts to be $k$.
Any such generalized partition $\mu$ can be expressed uniquely as an ordered pair $\mu = (\mu^+, \mu^-)$, where $\mu^+$ and $\mu^-$ are true partitions: in particular, $\mu^+$ consists of the positive parts of $\mu$, and $\mu^-$ is the partition obtained by negating and reversing the negative parts of $\mu$.
For example, if $\mu = (6,3,3,2,0,0,-1,-3,-5)$, then we have $\mu^+ = (6,3,3,2)$ and $\mu^- = (5,3,1)$.
Because the total number of parts is a fixed value $k$, there is no ambiguity in recovering the original $\mu$ from the pair $(\mu^+, \mu^-)$.
In this context of generalized partitions, we will write
\begin{align*}
\bmu &= (\mu_1, \ldots, \mu_r) = \Big( (\mu_1^+, \mu_1^-), \ldots, (\mu_r^+, \mu_r^-) \Big), \\
\bmu^+ & \coloneqq (\mu_1^+, \ldots, \mu_r^+),\\
    \bmu^- & \coloneqq (\mu_1^-, \ldots, \mu_r^-).
\end{align*}

\subsection*{Irreducible finite-dimensional representations of classical groups}

Let $\M_{p,q}$ denote the space of complex $p \times q$ matrices, and let $\M_n \coloneqq \M_{n,n}$.
The \emph{complex classical groups} consist of the general linear, orthogonal, and symplectic groups, defined as follows:
\begin{align*}
    \GL_k & \coloneqq \{ g \in \M_k : \det g \neq 0 \},\\
    \O_k & \coloneqq \{ g \in \GL_k : g^T = g^{-1} \},\\
    \Sp_{2k} & \coloneqq \{g \in \GL_{2k} : g^T J g = J \},
\end{align*}
where $J = \left[ \begin{smallmatrix}
    0 & I\\-I & 0
\end{smallmatrix} \right]$ and $I$ is the $k \times k$ identity matrix.
Note that these groups are precisely the complexifications $K$ arising in our motivating setting~\eqref{G info}.
For each classical group $K$, let $\widehat{K}$ denote the set of equivalence classes of finite-dimensional irreducible complex representations of $K$ with rational matrix coefficients.
It is well known that the elements of $\widehat{K}$ can be labeled by certain partitions $\la$ (or pairs thereof) as given below.
Throughout the paper, we write $U^\la$ to denote a model for the irreducible representation of $K$ labeled by $\lambda \in \widehat{K}$.
When treating a particular group $K$, we follow~\cite{HTW} in denoting the irreducible representations $U^\la$ as follows:
\begingroup
\renewcommand*{\arraystretch}{1.5}
\begin{equation}
\label{KC hat}
    \begin{array}{l|l|l}
        K & \widehat{K} & \text{Irrep.~$U^\la$} \\ \hline
       \O_k & \{ \la : \text{first two columns of $\la$ contain $\leq k$ boxes} \} & E^\la_k\\
       \GL_k & \{ \la = (\la^+, \la^-) : \ell(\la^+) + \ell(\la^-) \leq k \} & F^\la_k\\
       \Sp_{2k} & \{ \la : \ell(\la) \leq k \} & V^\la_{2k}
    \end{array}
\end{equation}
\endgroup
If $K = \GL_k$ or $\Sp_{2k}$, then $U^\la$ is the irreducible representation of $K$ with highest weight $\la$ (in standard coordinates).
For $K = \O_k$, which is not connected, the situation is more subtle; see~\cite{GW}*{pp.~438--9} or~\cite{FH}*{\S19.5} for a detailed construction of the representations $E^\la_k$.

\begin{dfn}
    \label{def:associated}
    Let $\la \in \widehat{\O}_k$.
   Its \emph{associated partition} $\overline{\la} \in \widehat{\O}_k$ is the partition obtained from $\la$ by changing the length of its first column from $\ell(\la)$ to $k - \ell(\la)$.
\end{dfn}

Note that taking the associated partition twice recovers the original partition $\la$.
The motivation behind Definition~\ref{def:associated} is the fact~\cite{FH}*{Ex.~19.23} that
\begin{equation}
    \label{tensor det}
    E^{\overline{\la}}_k \cong E^\la_k \otimes \mathbf{det}_k,
\end{equation}
where $\mathbf{det}_k$ denotes the one-dimensional determinant representation of $\O_k$.

\subsection*{Littlewood--Richardson coefficients}

If $U$ and $V$ are representations of a group $H$, then we define
\[
[U:V] := \dim \hom_{H}(U,V).
\]
Note that in general $[U:V]$ may be an infinite cardinal, but if $U$ and $V$ are finite-dimensional, then $[U:V]$ is a nonnegative integer.  
Moreover, if $U$ is irreducible and $V$ is completely reducible as an $H$-representation, then $[U:V]$ is the \emph{multiplicity} of $U$ in the isotypic decomposition of $V$.

Let $\la,\mu,\nu$ be partitions with length at most $n$.  
For all $m \geq n$, it is a standard fact that $\left[F^\la_n: F^\mu_n \otimes F^\nu_n \right] =\left[F^{\la}_m: F^{\mu}_m \otimes F^{\nu}_m \right]$.  
(See \cite{Macdonald}*{\S I.9}.)  
Thus, without ambiguity one can define $c^{\la}_{\mu \nu} := [F^\la_n: F^\mu_n \otimes F^\nu_n]$.  
Equivalently, we have $F^\mu_n \otimes F^\nu_n \cong \bigoplus_{\la} c^\la_{\mu \nu} F^\la_n$, where we use the shorthand $c U \coloneqq U^{\oplus c}$. The numbers $c^{\la}_{\mu \nu}$ are known as the \emph{Littlewood--Richardson coefficients}, and play a central role in algebraic combinatorics and classical representation theory~\cite{HoweLee}.

Recall that a \emph{semistandard tableau of shape $\la$} is the Young diagram of $\la$ in which the boxes are filled with entries from some totally ordered alphabet, say $\{1, \ldots, k\}$, such that the entries in each row are weakly increasing, and those in each column are strictly increasing.
The \emph{content} of a tableau $T$ is the $k$-tuple whose $i$th component is the number of occurrences of the entry $i$ in $T$.
The \emph{word} of $T$ is the sequence obtained by reading the entries of $T$ from right to left in each row, taking the rows from top to bottom.
We write $|T|$ to denote the number of boxes in $T$.

The Littlewood--Richardson coefficients carry the following well-known combinatorial interpretation.
If $\la$ and $\mu$ are partitions such that $\mu \subseteq \la$, then we write $\lambda / \mu$ to denote the \emph{skew diagram} consisting of the boxes in $\lambda$ which are not in $\mu$.
One can then speak of semistandard skew tableaux of shape $\la / \mu$.
We write $|\lambda / \mu|$ to denote the number of boxes in the skew diagram $\lambda / \mu$.
A semistandard skew tableau $T$ is called a \emph{Littlewood--Richardson (LR) tableau} if in the word of $T$ the number of occurrences of $i+1$ never exceeds the number of occurrences of $i$ for each $1 \leq i < k$.
We have the following \emph{Littlewood--Richardson rule}~\cite{Fulton}*{Prop.~3, p.~64}:
\begin{equation}
    \label{LR rule}
    c^\la_{\mu\nu} = \#\{\text{LR tableaux with shape $\la/\mu$ and content $\nu$}\}.
\end{equation}
In this paper, we will encounter~\eqref{LR rule} only in special cases where $\nu$ has length at most 2.
For this reason, we record the following definition of (horizontal) \emph{strips} and \emph{double strips}.

\begin{dfn}
\label{def:strip}
A skew diagram is said to be a \emph{strip} if it contains at most one box in each column.
A skew diagram is said to be a \emph{double strip} if it contains at most two boxes in each column.
\end{dfn}

The following special case of~\eqref{LR rule}, where $\nu = (m)$ for some $m \in \mathbb{N}$, is known as the \emph{Pieri rule}:
\begin{equation}
    \label{Pieri}
    c^\lambda_{\mu, (m)} = \begin{cases}
    1, & \lambda / \mu \text{ is a strip and } |\lambda / \mu| = m,\\
    0& \text{otherwise}.
\end{cases} 
\end{equation}
Similarly, in the special case where $\nu = (\ell,m)$, the LR tableaux in~\eqref{LR rule} contain only the entries 1 and 2, and therefore contain at most two boxes in each column.
Consequently,
\begin{equation}
    \label{LR double}
    \text{if $c^\la_{\mu, (\ell, m)} \neq 0$, then $\la / \mu$ is a double strip and $|\la / \mu| = \ell + m$.}
\end{equation}
As an example, let $\la = (6,5,3,1)$, $\mu = (5,2,1)$, and $\nu = (4,3)$.
We observe that $\la / \mu$ is a double strip containing 7 boxes, and therefore by~\eqref{LR double} it is possible that $c^\la_{\mu \nu}$ is nonzero.
In fact, by~\eqref{LR rule}, $c^\la_{\mu \nu}$ is the number of LR tableaux of shape $\la / \mu$ whose entries consist of four 1's and three 2's.
There are three such tableaux:
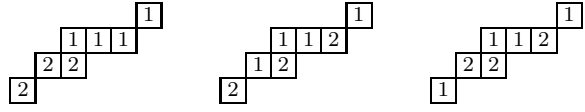
\ytableausetup{smalltableaux}
\[
\ytableaushort{\none\none\none\none\none1,\none\none111,\none22,2}
\qquad
\ytableaushort{\none\none\none\none\none1,\none\none112,\none12,2}
\qquad
\ytableaushort{\none\none\none\none\none1,\none\none112,\none22,1}
\]
Therefore $c^\la_{\mu \nu} = 3$.
The words of the tableaux above are $1111222$, $1211212$, and $1211221$, respectively.

By considering more than two tensor factors, one arrives at a natural generalization of the classical Littlewood--Richardson coefficients: 
in particular, given a partition $\la$, and a vector of partitions $\bmu$ as in \eqref{bold lambda}, where $\ell(\la) \leq n$ and each $\ell(\mu_i) \leq n$, we define the \emph{generalized Littlewood--Richardson coefficient}
\begin{equation*}
\label{LRC definition}
    \LRC{\la}{\bmu} \coloneqq \left[ F^\la_n:
    F^{\mu_1}_n \otimes \cdots \otimes F^{\mu_r}_n \right]_{\textstyle.}
\end{equation*}  
Given another vector $\bnu = (\nu_1, \ldots, \nu_s)$ with each $\ell(\nu_i) \leq n$, we will also write 
\begin{equation}
\label{LRC definition munu}
\LRC{\la}{\bmu \bnu} \coloneqq \left[ F^\la_n:
    F^{\mu_1}_n \otimes \cdots \otimes F^{\mu_r}_n \otimes F^{\nu_1}_n \otimes \cdots \otimes F^{\nu_s}_n \right]_{\textstyle.}
\end{equation}

\subsection*{Howe duality}
In this subsection, we recall the three \emph{dual pair settings} in which the irreducible representations of the classical groups $K$ arise.
For complete details, we refer the reader to Howe's paper~\cite{Howe89}, as well as the analytic perspective taken by Kashiwara--Vergne~\cite{KV}.

We summarize the Howe duality data in the table in Theorem~\ref{thm:Howe duality} below.
In each setting, the classical group $K$ acts naturally on a space $X$, which is a direct sum of copies of the defining representation of $K$, and of its contragredient representation in the case where $K = \GL_k$.
For the sake of concreteness, we realize $X$ as a space of matrices.
As usual, the action of $K$ on $X$ yields an action on the space $\C[X]$ of polynomial functions on $X$.

Let $\mathcal{D}(X)$ denote the Weyl algebra of polynomial-coefficient differential operators on $X$, and let $\mathcal{D}(X)^{K}$ denote the subalgebra of $K$-invariant operators.  
For each classical group $K$, there is a finite generating set of the associative algebra $\mathcal{D}(X)^K$ which spans a Lie subalgebra $\g'$ of $\mathcal{D}(X)^K$.
Thus, there exists a surjective homomorphism $\omega: U(\g') \longrightarrow \mathcal{D}(X)^{K}$, and we can view $\C[X]$ as a module for the Lie algebra $\g'$ via this homomorphism.
Similarly, if $U^\la$ is a finite-dimensional irreducible representation of $K$ with rational matrix coefficients,
then $\mathcal{D}(X)^{K}$ and hence $\g'$
acts on the multiplicity space
\[
\widetilde{U}^\la \coloneqq \hom_{K}(U^\la, \C[X]) \cong ((U^\la)^* \otimes \C[X])^{K}.
\]
Explicitly, the action of an invariant  differential operator $D\in \mathcal{D}(X)^{K}$ on $\widetilde{U}^\la$ is given by
$D\cdot(\sum_i u_i^* \otimes f_i)=\sum_{i} u_i^* \otimes (Df_i)$.
Further, the algebra $\mathcal{D}(X)^K$ acts irreducibly on the multiplicity space $\widetilde{U}^\la$, and hence $\g'$ does as well.

One can say more about the Lie algebra $\g'$ described above.
In fact, $\g'$ is the complexified Lie algebra of a real reductive Lie group $G'_\R$, with maximal compact subgroup $K'_\R$, such that $G'_\R/K'_\R$ is a Hermitian symmetric space.
Let $\k'$ denote the complexified Lie algebra of $K'_\R$, and let $K'$ denote the complexification of $K'_\R$.
Roughly speaking, a \emph{$(\g',K')$-module} is a 
complex vector space carrying representations of both $\g'$ and $K'$ such that  $K'$ acts locally finitely and the actions of $\g'$
and $K'$ are compatible; see~\cite{SchmidNotes}*{Def.~3.2.3} for details.
In each Howe duality setting described above, the $\mathfrak{k'}$-action integrates to a $K'$-action and hence 
$\widetilde{U}^\la$ can be viewed as a $(\g',K')$-module.
Further, $\widetilde{U}^\la$ is a highest weight $\g'$-module, which follows from work of Harish-Chandra~\cites{HC55,HC56} since $G'_\R/K'_\R$ is Hermitian symmetric.

\begin{thm}[Howe duality \cite{Howe89}, \cite{KV}]
\label{thm:Howe duality}
Assume one of the three settings in the following table:

\input{Table_Howe_duality}

\noindent We have the following multiplicity-free decomposition of $\C[X]$ as a $K \times (\g',K')$-module:
\begin{equation}
    \label{Howe decomp}
    \C[X] \cong \bigoplus_{\la \in \Sigma} U^\la \otimes \widetilde{U}^\la,
\end{equation}
where $\Sigma \coloneqq \{ \la \in \widehat{K} : \widetilde{U}^\la \neq 0\}$ and the irreducible modules $U^\la$ and $\widetilde{U}^\la$ are given in the following table:

\input{Table_Howe_duality_second}

\noindent Furthermore, there is an injective map $\la \mapsto \xi$ from $\Sigma$ to the set of dominant integral weights for $\k'$, such that as a $(\g',K')$-module, $ \widetilde{U}^\la$ is isomorphic to the simple $\g'$-module with highest weight $\xi$.

\end{thm}

In order to describe the structure of the modules $\widetilde{U}^\la$ in~\eqref{Howe decomp}, we recall some standard facts from the theory of Hermitian symmetric pairs $(\g', \k')$; see the exposition in~\cite{EHP}*{\S2.1}.
There exists a distinguished element $z$ in the center of $\k'$ such that $\operatorname{ad} z$ acts on $\g'$ with eigenvalues $0$ and $\pm 1$.
This yields a triangular decomposition 
$\g' = \p'_- \oplus \k' \oplus \p'_+$,
where $\p'_{\pm} = \{ x \in \g' : [z, x] = \pm x\}$.
In the dual pair settings listed in Theorem~\ref{thm:Howe duality}, where $\g'$ arises as the complexified Lie algebra of $G'_\R$, the realizations of $\g'$, $\k'$, and $\p'_+$ are given in Table~\ref{table:gkp} below.
We write ${\rm SM}_n$ (resp., ${\rm AM}_n$) to denote the space of symmetric (resp., skew-symmetric) $n \times n$ complex matrices.
For details of the explicit realizations of the real groups $G'_\R$, we refer the reader to Chapter 1 of~\cite{GW}.

\begin{table}[H]
    \centering
    \input{Table_gkp}
    \caption{Details in the Howe duality settings of Theorem~\ref{thm:Howe duality}}
    \label{table:gkp}
\end{table}

The subalgebra $\q' = \k' \oplus \p'_+$ is a maximal parabolic subalgebra of $\g'$, with Levi subalgebra $\k'$ and abelian nilradical $\p'_+$.  
Let $\xi$ be a dominant integral weight for $\k'$, and let $L^\xi$ be the finite-dimensional simple $\k'$-module with highest weight $\xi$.
Then $L^\xi$ is also a module for $\q'$, with $\p'_+$ acting by zero.  We define the \emph{generalized Verma module}
\begin{equation*}
    N^\xi \coloneqq U(\g') \otimes_{U(\q')} L^{\xi}.
\end{equation*}
Since $\p'_-$ is abelian, and since $\p'_- \cong (\p'_+)^*$ as $\k'$-modules, we can identify $U(\p'_-)$ with $S(\p'_-) \cong \C[\p'_+]$.  
By the Poincar\'e--Birkhoff--Witt theorem, we thus obtain
\[
N^\xi \cong \C[\p'_+] \otimes L^\xi
\]
as a $\k'$-module.
It turns out that inside a certain stable range (given in the following lemma), the $\g'$-modules $\widetilde{U}^\la$ are generalized Verma modules $N^\xi$ (where $\la \mapsto \xi$ is the map described in Theorem~\ref{thm:Howe duality}).

\begin{lemma}[$\k'$-structure in stable range]
    \label{lemma:stable}

In the stable range given in the table below, we have the following $\k'$-module structure of $\widetilde{U}^\la$, up to a central shift:

\input{Table_stable}

\end{lemma}

\begin{proof}
The stable range in each setting is given in~\cite{HTW}*{p.~1609}; see Remark~\ref{rem:stable} below.
The $\k'$-decompositions of $\C[\p'_+]$ are given in~\cite{GW}, in Theorem~5.6.7 and Corollaries~5.7.4 and~5.7.6; see also the treatment in~\cite{EHW}*{Thm.~3.1}, originally due to Schmid~\cite{Schmid}.
The $\k'$-decompositions of $\widetilde{U}^\la$ are given in~\cite{HTW}*{Thm.~3.2}.
\end{proof}

\begin{rem}
\label{rem:stable}
    Experts will observe that the stable range given in Lemma~\ref{lemma:stable} extends slightly farther than the stable range given in~\cite{HTW}*{p.~1609}.
    Specifically, the stable range in Lemma~\ref{lemma:stable} allows $k$ to be one less than the minimum $k$ given in~\cite{HTW}.
    Although not obvious \emph{a priori}, one can verify by examining the resolutions for $\widetilde{U}^\lambda$ constructed in~\cite{EW} that $\widetilde{U}^\lambda$ is a free $\C[\p'_+]$-module in the extended stable range given in our Lemma~\ref{lemma:stable}. 
    One can also verify this extended stable range using results of Schwarz~\cite{Schwarz} on cofree representations: for $\O_k$ see items 1 and 4 in Table 2, for $\GL_k$ see item 2 in Table 1a, and for $\Sp_{2k}$ see items 1 and $1'$ in Table 3 of~\cite{Schwarz}.
\end{rem}

The special case $k=1$ is especially important in this paper.
In the following lemma, since all partitions have length at most 1, we identify each partition $\delta$ (or pair thereof) with a single integer.
In particular, we identify a length-1 partition $(a)$ with the positive integer $a$.
The empty partition is identified with the integer $0$.
For $\GL_1$, we identify the pair $((a), 0)$ with the positive integer $a$, and the pair $(0,(b))$ with the negative integer $-b$.

\begin{lemma}
    \label{lemma:k=1}
    In the case $k=1$, with $\delta \in \widehat{K}$, we have the following decompositions of $\widetilde{U}^\delta$ as $\k'$-modules:

    \input{Table_k1}
   
\end{lemma}

\begin{proof}
     Recall~\cite{GW}*{Cor.~5.6.8} the following multiplicity-free decomposition of $\C[\M_{p,q}]$ as a representation of $\GL_p \times \GL_q$, acting via $((g,h) \cdot f)(A) = f(g^{T}Ah)$:
     \begin{equation}
     \label{GLp-GLq duality}
        \C[\M_{p,q}] \cong \bigoplus_{\mathclap{\substack{\nu:\\\ell(\nu) \leq \min\{p,q\}}}} F^\nu_p \otimes F^\nu_q.
     \end{equation}

     (a) Taking $p=1$ and $q=n$ in~\eqref{GLp-GLq duality}, we have
     \[
     \C[\M_{1,n}] \cong \bigoplus_{m \in \mathbb{N}} F^{(m)}_1 \otimes F^{(m)}_n
     \]
     as a representation of $\GL_1 \times \GL_n$.
     The action by $\GL_n$ differentiates to the action by $\k' = \gl_n$ in the Howe duality setting for $K = \O_1$ (see the first table in Theorem~\ref{thm:Howe duality}).
     Upon restricting from $\GL_1$ to $\O_1$, each one-dimensional representation $F^{(m)}_1$ is equivalent to $E^{(m \: {\rm mod} \: 2)}_1$.
     We therefore obtain the following decomposition of $\C[\M_{1,n}]$ under the action of $\O_1 \times \gl_n$:
     \begin{align*}
         \C[\M_{1,n}] & \cong \bigoplus_{m \in \mathbb{N}} E^{(m \: {\rm mod} \: 2)}_1 \otimes F^{(m)}_n \\
         & \cong \bigoplus_{\mathclap{\delta = \{0,1\}}} E^{\delta}_1 \otimes \left(\bigoplus_{\substack{m \in \mathbb{N}:\\ m \equiv \delta \: {\rm mod} \: 2}} F^{(m)}_n \right)_{\textstyle{,}}
     \end{align*}
     where the direct sum in parentheses gives the $\k'$-module structure of $\widetilde{E}^\delta_{2n}$.

     (b) Applying~\eqref{GLp-GLq duality} twice, we have
     \begin{align*}
     \C[\M_{1,p} \oplus \M_{1,q}] &\cong \C[\M_{1,p}] \otimes \C[\M_{1,q}] \\
     & \cong \left(\bigoplus_{\ell \in \mathbb{N}} F^{(\ell)}_1 \otimes F^{(\ell)}_p \right) \otimes \left(\bigoplus_{m \in \mathbb{N}} F^{(m)}_1 \otimes F^{(m)}_q \right) \\
     & \cong \bigoplus_{\ell,m \in \mathbb{N}} \left(F^{(\ell)}_1 \otimes F^{(m)}_1 \right) \otimes \left(F^{(\ell)}_p \otimes F^{(m)}_q \right)
     \end{align*}
     as a representation of $(\GL_1 \times \GL_1) \times (\GL_p \times \GL_q)$.
     The action by $\GL_p \times \GL_q$ differentiates to the action by $\k' = \gl_p \oplus \gl_q$ in the Howe duality setting for $K = \GL_1$ (see the first table in Theorem~\ref{thm:Howe duality}).
     Upon restricting from $\GL_1 \times \GL_1$ to the diagonal subgroup $\{(g,g^{-1}) : g \in \GL_1 \} \cong \GL_1$, each one-dimensional representation $F^{(\ell)}_1 \otimes F^{(m)}_1$ is equivalent to $F^{(\ell-m)}_1$.
     (Note that we choose the diagonal embedding $(g, g^{-1})$ due to the action of $\GL_1$ on $\M_{1,p} \oplus \M_{1,q}$ given in Theorem~\ref{thm:Howe duality}: for general $k$ we have $g \cdot (A,B) = (gA, (g^{-1})^{T} B)$, and since $k=1$ the transpose has no effect.)
     We therefore obtain the following decomposition of $\C[\M_{1,p} \oplus \M_{1,q}]$ under the action of $\GL_1 \times (\gl_p \oplus \gl_q)$:
     \begin{align*}
         \C[\M_{1,p} \oplus \M_{1,q}] & \cong \bigoplus_{\ell,m \in \mathbb{N}} F^{(\ell - m)}_1 \otimes \left(F^{(\ell)}_p \otimes F^{(m)}_q \right) \\
         & \cong \bigoplus_{\delta \in \mathbb{Z}} F^\delta_1 \otimes \left( \bigoplus_{\substack{m \in \mathbb{N}: \\ m + \delta \geq 0}} F^{(m + \delta)}_p \otimes F^{(m)}_q \right)_{\textstyle{,}}
     \end{align*}
     where the direct sum in parentheses gives the $\k'$-module structure of $\widetilde{F}^\delta_{p,q}$.

     (c) 
     Taking $p=2$ and $q=n$ in~\eqref{GLp-GLq duality}, we have
     \[
     \C[\M_{2,n}] \cong \bigoplus_{\mathclap{\substack{\ell,m \in \mathbb{N}: \\
     \ell \geq m}}} F^{(\ell,m)}_2 \otimes F^{(\ell,m)}_n
     \]
     as a representation of $\GL_2 \times \GL_n$.
     The action by $\GL_n$ differentiates to the action by $\k' = \gl_n$ in the Howe duality setting for $K = \Sp_2$ (see the first table in Theorem~\ref{thm:Howe duality}).
     Upon restricting from $\GL_2$ to $\Sp_2 \cong {\rm SL}_2$, each representation $F^{(\ell,m)}_2$ is equivalent to $V^{(\ell-m)}_2$.
     We therefore obtain the following decomposition of $\C[\M_{2,n}]$ under the action of $\Sp_2 \times \gl_n$:
     \begin{align*}
         \C[\M_{2,n}] & \cong \bigoplus_{\ell \geq m} V^{(\ell-m)}_2 \otimes F^{(\ell,m)}_n \\
         & \cong \bigoplus_{\delta \in \mathbb{N}} V^\delta_2 \otimes \left( \bigoplus_{m \in \mathbb{N}} F^{(m+\delta,m)}_n \right)_{\textstyle{,}}
     \end{align*}
     where the direct sum in parentheses gives the $\k'$-module structure of $\widetilde{V}^\delta_{2n}$.
\end{proof}

\begin{rem}
    In the case $K = \O_1$ in Lemma~\ref{lemma:k=1}, the representation $\C[X] = \C[\M_{1,n}] \cong \C[x_1, \ldots, x_n]$ is the underlying Harish-Chandra module of the \emph{oscillator representation} of the metaplectic group, i.e., the double-cover of $G' = \Sp(2n,\R)$.
    For $\widehat{K}= \{0,1\}$, the empty partition $0$ labels the trivial representation $\bm{1} \coloneqq E^0_1$ of $\O_1$, while $1$ labels the sign representation $\mathbf{sgn} \coloneqq E^1_1$.
    It follows from Theorem~\ref{thm:Howe duality} that $\C[x_1, \ldots, x_n]$ decomposes into two components under the action of $\O_1 \times \sp_{2n}$:
\[
\C[x_1, \ldots, x_n] \cong \left(\bm{1} \otimes \widetilde{E}^0_{2n}\right) \oplus \left(\mathbf{sgn} \otimes \widetilde{E}^1_{2n}\right).
\]
    These two components consist of the polynomials of even and odd degree, respectively.
\end{rem}

\subsection*{Direct sum embeddings}

Let $k_1, \ldots, k_r$ be positive integers summing to $k$.
Inside each classical group $K$, one can naturally embed a direct sum $M$ of smaller classical groups (of the same type) as a subgroup, via block diagonal matrices; see the table~\eqref{G info}.
In each case, upon restricting from $K$ to $M$, the matrix space $X$ decomposes into a direct sum of vertically stacked blocks as follows:
\begingroup
\renewcommand*{\arraystretch}{1.5}
\begin{equation}
\label{W block decomp}
    \begin{array}{l|l|l}
        M & X & \C[X]\\ \hline
        \prod_{i=1}^r \O_{k_i} \quad\quad & \bigoplus_{i=1}^r \M_{k_i,n} & \bigotimes_{i=1}^r \C[\M_{k_i,n}] \\
        \prod_{i=1}^r \GL_{k_i} \quad\quad & \bigoplus_{i=1}^r (\M_{k_i,p} \oplus \M_{k_i,q}) & \bigotimes_{i=1}^r \C[\M_{k_i,p} \oplus \M_{k_i,q}]\\
        \prod_{i=1}^r \Sp_{2k_i} \quad\quad & \bigoplus_{i=1}^r \M_{2k_i,n} & \bigotimes_{i=1}^r \C[\M_{2k_i,n}]
    \end{array}
\end{equation}
\endgroup
In the action of $M$ on $\C[X]$ in~\eqref{W block decomp}, each $\O_{k_i}$ (resp., $\GL_{k_i}$ or $\Sp_{2k_i}$) acts on the tensor factor $\C[\M_{k_i,n}]$ (resp., $\C[\M_{k_i,p} \oplus \M_{k_i,q}]$ or $\C[\M_{2k_i, n}]$) exactly as in Theorem~\ref{thm:Howe duality}, and acts trivially on every other tensor factor.
Therefore, by replacing $k$ by $k_i$ (in the table in Theorem~\ref{thm:Howe duality}) for each tensor factor in~\eqref{W block decomp}, we have the following multiplicity-free decomposition of $\C[X]$ as a module for $M \times (\g')^{\oplus r}$:
\begin{equation}
    \label{LC decomp}
    \C[X] \cong \bigoplus_{\substack{\bmu = (\mu_1, \ldots, \mu_r):\\\mu_i \in \Sigma_i}} \underbrace{\left(\bigotimes_{i=1}^r U^{\mu_i}\right)}_{W^{\bmu}} \otimes \underbrace{\left( \bigotimes_{i=1}^r \widetilde{U}^{\mu_i}\right)}_{\widetilde{W}^{\bmu}} \raisebox{-2ex}{,}
\end{equation}
where $\Sigma_i$ denotes the set $\Sigma$ in Theorem~\ref{thm:Howe duality} upon replacing $k$ by $k_i$.
Whereas we used the letter $U$ to denote an irreducible representation of $K$ (decorated with a tilde to denote the corresponding infinite-dimensional $\g'$-module), we use $W$ to denote an irreducible representation of the subgroup $M$ (decorated with a tilde to denote the corresponding infinite-dimensional $(\g')^{\oplus r}$-module).

\subsection*{Seesaw reciprocity}

Compare the two multiplicity-free decompositions of $\C[X]$ given in~\eqref{Howe decomp} and~\eqref{LC decomp}:
\begin{align*}
    \C[X] &\cong \bigoplus_{\la \in \Sigma} U^\la \otimes \widetilde{U}^\la && \text{as a module for $K \times \g'$}\\
    \C[X] & \cong \bigoplus_{\mathclap{\substack{\bmu \\ \in \Sigma_1 \times \cdots \times \Sigma_r}}} W^{\bmu} \otimes \widetilde{W}^{\bmu} && \text{as a module for $M \times (\g')^{\oplus r}$}.
\end{align*}
Following Kudla~\cite{Kudla}, we say that $K \times \g'$ and $M \times (\g')^{\oplus r}$ form a \emph{seesaw pair} on the space $\C[X]$, which we display as
\begin{equation}
\label{seesaw paradigm}
\begin{array}{ccl}
    K & \times & \g'  \\[1ex]
    \rotatebox[origin=c]{90}{$\subset$} & & \rotatebox[origin=c]{90}{$\supset$} \\[1ex]
    M & \times & (\g')^{\oplus r}.
\end{array}
\end{equation}
Because the $K$- and $M$-actions are the complexifications of actions by compact groups $K_\R$ and $M_\R$, both actions on the left-hand side of the seesaw pair~\eqref{seesaw paradigm} are completely reducible, so the actions of the Lie algebras on the right-hand side are completely reducible as well (see \cites{Howe89, HK}).
The term ``seesaw" describes the reciprocity of branching multiplicities
\begin{equation}
\label{eq:seesaw}
     b^\la_{\bmu} = \left[ W^{\bmu} : U^\la \right] = \left[ \widetilde U^\la : \widetilde {W}^{\bmu} \right]_{\textstyle{.}}
\end{equation}
Thus, the restriction problem (for finite-dimensional modules) from $K$ to $M$ is equivalent to a restriction problem (for infinite-dimensional modules) from $(\g')^{\oplus r}$ to $\g'$.

\section{Stable branching rule from $K$ to $M$}
\label{sec:stable}

Our main result in this section, namely Theorem~\ref{thm:main result}, is a detailed statement of Theorem~\ref{theorem 1 in intro}.
In it, we express the branching multiplicity $b^\la_{\bmu} \coloneqq [W^{\bmu} : U^\la]$ as a sum of generalized Littlewood--Richardson coefficients.
This rule is valid as long as the parameters $\lambda$ and $\bmu$ lie inside a certain stable range, given explicitly in Theorem~\ref{thm:main result} below.
Roughly speaking, the maximum valid length of $\la$ and of each $\mu_i$ depends on the smallest parameter $k_i$.
The stability assumption allows us to prove the following theorem using nothing more than seesaw reciprocity along with the $\k'$-decompositions from Lemma~\ref{lemma:stable}.

\begin{thm}[cf.~Theorem~\ref{theorem 1 in intro}]
\label{thm:main result}

Let $M \subset K$ be a direct sum embedding in a classical group, as in~\eqref{G info}, where $k_1, \ldots, k_r$ are positive integers summing to $k$.
Let $b^\la_{\bmu} \coloneqq [W^{\bmu} : U^\la]$ denote the branching multiplicity from $K$ to $M$.
Recall the generalized Littlewood--Richardson coefficients $c^\la_{\bmu \bnu}$ defined in~\eqref{LRC definition munu}.

\begin{enumerate}[label=\textup{(\alph*)}]
    
\item \textup{Stable branching rule from $K = \O_k$ to $M = \O_{k_1} \times \cdots \times \O_{k_r}$:}

Let $n$ be a positive integer such that $n \leq \frac{1}{2}(1+ \min_i\{k_i\})$.
Let $\la$ be a partition such that $\ell(\la) \leq n$.
Let $\bmu = (\mu_1, \ldots, \mu_r)$ be a vector of partitions, such that each $\ell(\mu_i) \leq n$.
Then we have
\[
    b_{\bmu}^\la = \sum_{\bnu} \: \LRC{\la}{\bmu\bnu},
\]
where $\bnu = (\nu_1, \ldots, \nu_{r-1})$ is a vector in which each partition $\nu_j$ has even rows, with $\ell(\nu_j) \leq n$.

\item \textup{Stable branching rule from $K = \GL_k$ to $M = \GL_{k_1} \times \cdots \times \GL_{k_r}$:}
    
     Let $p$ and $q$ be positive integers such that $p+q \leq 1+\min_i\{k_i\}$.
     Let $\la = (\la^+, \la^-)$ be such that $\ell(\la^+) \leq p$ and $\ell(\la^-) \leq q$.
     Let $\bmu = (\mu_1, \ldots, \mu_r) =((\mu^+_1, \mu^-_1), \ldots, (\mu^+_r, \mu^-_r))$ be such that each $\ell(\mu^+_i) \leq p$ and $\ell(\mu^-_i) \leq q$.
     Then we have
     \[
     b_{\bmu}^\la = \sum_{\bnu} \LRC{\la^+}{\bmu^+\bnu} \:\LRC{\la^-}{\bmu^-\bnu},
\]
where the sum ranges over all vectors $\bnu=(\nu_1,\dots,\nu_{r-1})$ such that each $\ell(\nu_j)\leq \min\{p,q\}$.

\item \textup{Stable branching rule from $K = \Sp_{2k}$ to $M = \Sp_{2k_1} \times \cdots \times \Sp_{2k_r}$:}

Let $n$ be a positive integer such that $n \leq 1 + \min_i\{k_i\}$.
Let $\la$ be a partition with $\ell(\la) \leq n$.
Let $\bmu = (\mu_1, \ldots, \mu_r)$ be a vector of partitions, such that each $\ell(\mu_i) \leq n$.
Then we have
\[
    b_{\bmu}^\la = \sum_{\bnu} \LRC{\la}{\bmu\bnu},
\]
where $\bnu = (\nu_1, \ldots, \nu_{r-1})$ is a vector in which each partition $\nu_j$ has even columns, with $\ell(\nu_j) \leq n$.

\end{enumerate}
       
\end{thm}

\begin{proof}\

(a) By seesaw reciprocity~\eqref{eq:seesaw}, we have
\begin{equation}
\label{b Ok proof}
    b^\la_{\bmu} = \left[ \widetilde{E}^\la_{2n} : \bigotimes_{i=1}^r \widetilde{E}^{\mu_i}_{2n} \right]_{\textstyle{,}}
\end{equation}
where the modules on the right-hand side are for $\g' = \sp_{2n}$.
The assumption on $n$ implies that each of the parameters $k_i$, as well as $k$, are in the stable range (see the table in Lemma~\ref{lemma:stable}).
Further, the assumptions on $\ell(\la)$ and $\ell(\mu_i)$ guarantee that the desired representations occur in the Howe duality setting of Theorem~\ref{thm:Howe duality}.
It follows from Lemma~\ref{lemma:stable} that as $\k' = \gl_n$-modules,
\begin{align*}
    \widetilde{E}^{\mu_i}_{2n} & \cong \C[{\rm SM}_n] \otimes F^{\mu_i}_n, \qquad 1 \leq i \leq r,\\
    \widetilde{E}^\la_{2n} & \cong \C[{\rm SM}_n] \otimes F^\la_n.
\end{align*}
Substituting these into~\eqref{b Ok proof}, we have
\begin{equation}
    \label{subbed b Ok proof}
    b^\la_{\bmu} = \left[ \C[{\rm SM}_n] \otimes F^\la_n : \bigotimes_{i=1}^r \C[{\rm SM}_n] \otimes F^{\mu_i}_n \right]_{\textstyle{.}}
\end{equation}
By Table~\ref{table:gkp}, the $\k'$-module $\bigotimes_{i=1}^r \C[{\rm SM}_n] \otimes F^{\mu_i}_n$ becomes
\begin{align*}
    & \phantom{=}  \left(\bigotimes_{j=1}^{r-1} \C[{\rm SM}_n] \right) \otimes \C[{\rm SM}_n] \otimes \bigotimes_{i=1}^r F^{\mu_i}_n\\
    &\cong   \bigotimes_{j=1}^{r-1} \left( \bigoplus_{\nu_j} F^{\nu_j}_n \right) \otimes \C[{\rm SM}_n] \otimes \bigotimes_{i=1}^r F^{\mu_i}_n & \text{each $\nu_j$ with even rows, and $\ell(\nu_j) \leq n$}\\
    & \cong  \bigoplus_{\bm{\nu}} \C[{\rm SM}_n] \otimes F^{\mu_1}_n \otimes \cdots \otimes F^{\mu_r}_n \otimes F^{\nu_1}_n \otimes \cdots \otimes F^{\nu_{r-1}} & \text{summing over all $\bm{\nu} = (\nu_1, \ldots, \nu_{r-1})$} \\
    & \cong \bigoplus_{\bm{\nu}} \C[{\rm SM}_n] \otimes 
 \bigoplus_\ga \LRC{\ga}{\bmu \bm{\nu}}  F^\ga_n & \text{by~\eqref{LRC definition munu}} \\
    & \cong \bigoplus_\ga \bigoplus_{\bm{\nu}} \LRC{\ga}{\bmu \bm{\nu}}  \left( \C[{\rm SM}_n] \otimes F^\ga_n \right).
\end{align*}
Upon substituting this into~\eqref{subbed b Ok proof}, it is clear that $b^\la_{\bmu} = \sum_{\bm{\nu}} \LRC{\la}{\bmu \bm{\nu}}$, where the sum ranges over those vectors $\bm{\nu}$ satisfying the conditions stated in the theorem.
This completes the proof of part (a).

(b) The argument is the same as in part (a), \emph{mutatis mutandis}.
In this case, by seesaw reciprocity~\eqref{eq:seesaw} we have 
\begin{equation}
    \label{reciprocity GL}
    b_{\bmu}^\la =\left[\widetilde{F}^\la_{p,q}:\bigotimes_{i=1}^r \widetilde{F}^{\mu_i}_{p,q}\right]_{\textstyle{.}}
\end{equation}
The assumptions on $p$ and $q$ imply that each of the parameters $k_i$, as well as $k$, are in the stable range (see the table in Lemma~\ref{lemma:stable}).
Further, the assumptions on $\ell(\la^\pm)$ and $\ell(\mu^\pm_i)$ guarantee that the desired representations occur in the Howe duality setting of Theorem~\ref{thm:Howe duality}.
It follows from Lemma~\ref{lemma:stable} that as $\k' = \gl_p \times \gl_q$-modules,
\begin{align*}
    \widetilde{F}^{\mu_i}_{p,q} &\cong \C[\M_{p,q}] \otimes F^{\mu^+_i}_p \otimes F^{\mu^-_i}_q, \qquad 1 \leq i \leq r,\\
    \widetilde{F}^{\la}_{p,q} &\cong \C[\M_{p,q}] \otimes F^{\la^+}_p \otimes F^{\la^-}_q
\end{align*}
Substituting these expressions into~\eqref{reciprocity GL}, we obtain
\begin{equation}
\label{subbed b GLk proof}
    b_{\bmu}^\la =  \left[\C[\M_{p,q}]\otimes \left(F^{\la^+}_{p}\otimes F^{\la^-}_{q} \right) : \bigotimes_{i=1}^r \C[\M_{p,q}]\otimes \left(F^{\mu_i^+}_{p}\otimes F^{\mu_i^-}_{q}\right)\right]_{\textstyle{.}}
\end{equation}
By Table~\ref{table:gkp}, the $\k'$-module $\bigotimes_{i=1}^r \C[\M_{p,q}]\otimes \left(F^{\mu_i^+}_{p}\otimes F^{\mu_i^-}_{q}\right)$ becomes
\begin{align*}
    & \phantom{=} \left(\bigotimes_{j=1}^{r-1} \C[\M_{p,q}] \right) \otimes \C[\M_{p,q}] \otimes \left(\bigotimes_{i=1}^r F^{\mu^+_i}_p \otimes F^{\mu^-_i}_q \right) \\  
    & \cong \bigotimes_{j=1}^{r-1} \left( \bigoplus_{\nu_j} F^{\nu_j}_p \otimes F^{\nu_j}_q \right) \otimes \C[\M_{p,q}] \otimes \left(\bigotimes_{i=1}^r F^{\mu^+_i}_p \otimes F^{\mu^-_i}_q \right) \text{ where each $\ell(\nu_j) \leq \min\{p,q\}$}\\[1ex]
    & \cong \left(\bigoplus_{\bnu} \left(\bigotimes_{j=1}^{r-1} F^{\nu_j}_p \otimes \bigotimes_{j=1}^{r-1} F^{\nu_j}_q \right)\right) \otimes \C[\M_{p,q}] \otimes \left( \bigotimes_{i=1}^r F^{\mu^+_i}_p \right) \otimes \left( \bigotimes_{i=1}^r F^{\mu^-_i}_q \right) \\[1ex]
    & \cong \bigoplus_{\bnu} \C[\M_{p,q}] \otimes  \left( \bigotimes_{i=1}^r F^{\mu^+_i}_p \otimes \bigotimes_{j=1}^{r-1} F^{\nu_j}_p \right) \otimes \left( \bigotimes_{i=1}^r F^{\mu^-_i}_q \otimes \bigotimes_{j=1}^{r-1} F^{\nu_j}_q \right) \\[1ex]
    & \cong \bigoplus_{\bnu} \C[\M_{p,q}] \otimes \left( \bigoplus_{\ga^+} \LRC{\ga^+}{\bmu^+ \bnu} \left(F^{\ga^+}_p\right)\right) \otimes \left( \bigoplus_{\ga^-} \LRC{\ga^-}{\bmu^- \bnu} \left(F^{\ga^-}_q\right)\right) \nonumber\\[1ex]
    & \cong \bigoplus_{\substack{\ga = \\ (\ga^+, \ga^-)}} \bigoplus_{\bnu} \LRC{\ga^+}{\bmu^+ \bnu} \: \LRC{\ga^-}{\bmu^- \bnu} \left(\C[\M_{p,q}] \otimes F^{\ga^+}_p \otimes F^{\ga^-}_q \right)_{\textstyle{.}}
\end{align*}
Upon substituting this into~\eqref{subbed b GLk proof}, it is clear that $b_{\bmu}^\la = \sum_{\bnu} \LRC{\la^+}{\bmu^+ \bnu} \LRC{\la^-}{\bmu^- \bnu}$, where $\bm{\nu}$ satisfies the conditions stated in the theorem.
This completes the proof of part (b).

(c) The proof is identical to part (a), this time using the $K = \Sp_{2k}$ case in Theorem~\ref{thm:Howe duality}, Lemma~\ref{lemma:stable}, and Table~\ref{table:gkp}.
\end{proof}

\section{Branching rule from $K$ to the minimal $M$}

The main result in this section, namely Theorem~\ref{thm:K to M}, is a detailed statement of Theorem~\ref{theorem 2 in intro}.
This result is a tableau-theoretic branching rule from $K$ to its minimal direct sum embedding $M$.
Unlike Theorem~\ref{thm:main result}, this branching rule is fully general, valid for all irreducible finite-dimensional rational representations of $K$ and $M$.

Throughout this section, $M$ is given by the special case in~\eqref{G info} where $k_1 = \cdots = k_r = 1$.
The irreducible representations $W^{\bm{\delta}}$ of $M$ are labeled by the following $k$-tuples~$\bm{\delta}$:
\begingroup
\renewcommand*{\arraystretch}{1.5}
\begin{equation}
\label{M hat}
    \begin{array}{l|l|l}
        K & M & \bm{\delta} \in \widehat{M} \\ \hline

        \O_k & (\O_1)^k & \bm{\delta} \in \{0,1\}^k\\
        \GL_k & (\GL_1)^k & \bm{\delta} \in \mathbb{Z}^k\\
        \Sp_{2k} & (\Sp_2)^k & \bm{\delta} \in \mathbb{N}^k     
    \end{array}
\end{equation}
\endgroup
We now use $\bm{\delta}$ rather than $\bm{\mu}$, since the integers $\delta_i$ will be viewed as differences between frequencies of entries in certain tableaux.
Let $\SSYT(k)$ denote the set of semistandard tableaux whose entries are taken from the ordered alphabet $1 < \cdots < k$.
Likewise, let $\SSYT(1, \bar{1}, \ldots, k, \bar{k})$ denote the set of semistandard tableaux whose entries are taken from the ordered alphabet $1 < \bar{1} < \cdots < k < \bar{k}$.

\begin{dfn}[$K$-tableaux and their weights]\
\label{def:KC-tableaux}
    
    \begin{enumerate}[label=(\alph*)]
        \item A tableau $T \in \SSYT(k)$ is an \emph{$\O_k$-tableau} if, for all $1 \leq i \leq k$, we have
        \[
        \#\{ \text{boxes in first two columns of $T$ with entry $\leq i$}\} \leq i.
        \]
        We define $\mathbf{wt}(T)$ to be the vector $(w_1, \ldots, w_k) \in \widehat{M} = \{0,1\}^k$ such that 
        \[
        w_i = \#\{\text{boxes in $T$ with entry $i$}\} \: {\rm mod} \: 2.
        \]
        
        \item A pair $(T^+, T^-) \in \SSYT(k) \times \SSYT(k)$ is a \emph{$\GL_k$-tableau} if, for all $1 \leq i \leq k$, we have
        \[
        \#\{\text{boxes in first column of $T^+$ or $T^-$ with entry $\leq i$}\} \leq i.
        \]
        We define $\mathbf{wt}(T^+, T^-)$ to be the vector $(w_1, \ldots, w_k) \in \widehat{M} = \mathbb{Z}^k$ such that
        \[
        w_i = \#\{\text{boxes in $T^+$ with entry $i$}\} - \#\{\text{boxes in $T^-$ with entry $i$}\}.
        \]

        \item A tableau $T \in \SSYT(1, \bar{1}, \ldots, k, \bar{k})$ is an \emph{$\Sp_{2k}$-tableau} if, for all $1 \leq i \leq k$, we have
        \[
        \#\{\text{boxes in first column of $T$ with entry $\leq \bar{\imath}$}\} \leq i.
        \]
        Moreover, we say that $T$ is an \emph{$\Sp_{2k}$-ballot tableau} if, as one reads the word of $T$, the number of $\bar{\imath}$'s never exceeds the number of $i$'s, for each $1 \leq i \leq k$.
        
        If $T$ is an $\Sp_{2k}$-ballot tableau, then we define $\mathbf{wt}(T)$ to be the vector $(w_1, \ldots, w_k) \in \widehat{M} = \mathbb{N}^k$ such that
        \[
        w_i = \#\{\text{boxes in $T$ with entry $i$}\} - \#\{\text{boxes in $T$ with entry $\bar{\imath}$}\}.
        \]
    \end{enumerate}
\end{dfn}

Note that Definition~\ref{def:KC-tableaux} can be restated as follows: $T$ is a $K$-tableau if and only if, upon restricting $T$ to those entries $\leq i$ (ignoring the overlines for $\Sp_{2k}$), the resulting shape lies in the set $\widehat{\O}_i$ or $\widehat{\GL}_i$ or $\widehat{\Sp}_{2i}$.
We observe that an $\Sp_{2k}$-ballot tableau cannot contain any entries $\bar{\imath}$ in its first row, and therefore cannot contain the entry $\bar{1}$ at all.

For $\la \in \widehat{K}$ and $\bm{\delta} \in \widehat{M}$, we define the following sets of $K$-tableaux refined by their shape and weight:
\begin{align}
    \label{T(KC)}
    \begin{split}
        \mathcal{T}(\O_k)^\la_{\bm{\delta}} &\coloneqq \{\text{$\O_k$-tableaux $T$ of shape $\la$, such that $\mathbf{wt}(T) = \bm{\delta}$}\},\\
        \mathcal{T}(\GL_k)^\la_{\bm{\delta}} &\coloneqq \{ \text{$\GL_k$-tableaux $(T^+, T^-)$ of shape $(\la^+, \la^-)$, such that $\mathbf{wt}(T^+,T^-) = \bm{\delta}$}\},\\
        \mathcal{T}(\Sp_{2k})^\la_{\bm{\delta}} &\coloneqq \{ \text{$\Sp_{2k}$-ballot tableaux $T$ of shape $\la$, such that $\mathbf{wt}(T) = \bm{\delta}$}\}.
    \end{split}
\end{align}

\begin{thm}[cf.~Theorem~\ref{theorem 2 in intro}]
    \label{thm:K to M}

    Let $M \subseteq K$ as in~\eqref{M hat}.
    Let $\la \in \widehat{K}$ and $\bm{\delta} = (\delta_1, \ldots, \delta_k) \in \widehat{M}$.
    Let $b^\la_{\bm{\delta}} \coloneqq [W^{\bm{\delta}} : U^\la]$ denote the branching multiplicity from $K$ to $M$.
    We have
\[
    b^\la_{\bm{\delta}} = \# \mathcal{T}(K)^\la_{\bm{\delta}}.
\]

\end{thm}

Before proving Theorem~\ref{thm:K to M}, we present the following examples.

\begin{ex}
    Let $K = \O_5$, and let $\la = (2,2)$ and $\bm{\delta} = \mathbf{0}$.
    By Theorem~\ref{thm:K to M} and~\eqref{T(KC)}, the branching multiplicity $b^{(2,2)}_{\mathbf{0}}$ equals the number of $\O_5$-tableaux of shape $(2,2)$ and weight $\mathbf{0}$.
    Each such semistandard tableau $T$ is filled with entries from $\{1, \ldots, 5\}$, such that there are no more than $i$ entries less than or equal to $i$ in the first two columns, and every entry occurs an even number of times.
    There are five such tableaux:
    \[
    \ytableausetup{centertableaux, boxsize=1.3em}
    b^{(2,2)}_{\mathbf{0}} = \#
    \mathcal{T}(\O_5)^{(2,2)}_{\mathbf{0}} = \#\left\{ \; \ytableaushort{22,44}_{\textstyle{,}} \quad 
    \ytableaushort{22,55}_{\textstyle{,}} \quad 
    \ytableaushort{33,44}_{\textstyle{,}} \quad 
    \ytableaushort{33,55}_{\textstyle{,}} \quad 
    \ytableaushort{44,55}\;
    \right\} = 5.
    \]
    This agrees with the example in Section 4.2 of~\cite{Frohmader}, upon setting $q=1$ in that example.
\end{ex}

\begin{ex}
    Let $K = \GL_4$, and let $\la = (2,1,-2,-2)$ and $\bm{\delta} = (2, -1, -2, 0)$.
    Then we have $\la^+ = (2,1)$ and $\la^- = (2,2)$.
    By Theorem~\ref{thm:K to M}, the branching multiplicity $b^\la_{\bm{\delta}}$ equals the number of $\GL_4$-tableaux of shape $((2,1), \: (2,2))$ and weight $(2, -1, -2, 0)$.
    There is exactly one such $\GL_4$-tableau:
    \[
    b^{(2,1,-2,-2)}_{(2,-1,-2,0)} = \# \mathcal{T}(\GL_4)^{(2,1,-2,-2)}_{(2,-1,-2,0)} = \# \left\{ \; \left( \: \begin{ytableau} 1&1 \\ 4\end{ytableau}_{\textstyle{,}} \; 
    \begin{ytableau}2&3 \\ 3&4 \end{ytableau} \: \right)
    \;
    \right\} = 1.
    \]
    Since $M = \GL_1 \times \cdots \times \GL_1$ is the standard torus in $\GL_k$, the numbers $b^\la_{\bm{\delta}}$ give the weight multiplicities in the rational representation $F^\la_k$.
    In this way, the numbers $b^\la_{\bm{\delta}}$ generalize the Kostka numbers $K_{\la \mu}$, which count the number of semistandard tableau of shape $\la$ and content $\mu$, and which give the dimension of the $\mu$-weight space in a polynomial representation $F^\la_k$.
    The tableau pairs in $\mathcal{T}(\GL_k)^\la_{\bm{\delta}}$ can also be viewed as a generalization of Gelfand--Zeitlin tableaux~\cite{GW}*{Cor.~8.1.7}, for rational (rather than polynomial) representations of $\GL_k$.
    See also the character-theoretic approach to rational $\GL_k$-tableaux taken by Stembridge~\cite{Stembridge}.
\end{ex}

\begin{ex}
    Let $K = \Sp_6$, and let $\la = (2,2)$ and $\bm{\delta} = \mathbf{0}$.
    By Theorem~\ref{thm:K to M} and~\eqref{T(KC)}, the branching multiplicity $b^{(2,2)}_{\mathbf{0}}$ equals the number of $\Sp_6$-ballot tableaux of shape $(2,2)$ and weight $\mathbf{0}$.
    Each such semistandard tableau $T$ is filled with entries from $\{1, \bar{1}, 2, \bar{2}, 3, \bar{3}\}$, satisfies the condition in Definition~\eqref{def:KC-tableaux}(c), and has an equal number of $i$'s and $\bar{\imath}$'s for each $1 \leq i \leq 3$.
    There are three such tableaux:
    \[
    b^{(2,2)}_{\mathbf{0}} = \#
    \mathcal{T}(\Sp_6)^{(2,2)}_{\mathbf{0}} = \#\left\{ \; \begin{ytableau}2&2 \\ \raisebox{-2pt}{$\overline{2}$} & \raisebox{-2pt}{$\overline{2}$} \end{ytableau}_{\textstyle{,}} \quad 
    \begin{ytableau}2&3 \\ \raisebox{-2pt}{$\overline{2}$} & \raisebox{-2pt}{$\overline{3}$} \end{ytableau}_{\textstyle{,}} \quad
    \begin{ytableau}3&3 \\ \raisebox{-2pt}{$\overline{3}$} & \raisebox{-2pt}{$\overline{3}$} \end{ytableau}
    \;
    \right\} = 3.
    \]
\end{ex}

\begin{lemma}
    \label{lemma:Ok assoc rule}
    Let $\la \in \widehat{\O}_k$, let $\mu \in \widehat{\O}_{k-1}$, and let $\delta \in \widehat{\O}_1 = \{0,1\}$.
    Let $\overline{\la} \in \widehat{\O}_k$, $\overline{\mu} \in \widehat{\O}_{k-1}$, and $\overline{\delta} \in \widehat{\O}_1$ be the associated partitions, in the sense of Definition~\ref{def:associated}.
    We have the following:
   \begin{enumerate}[label=\textup{(\arabic*)}]
    
    \item $
    b^{\la}_{(\mu,\delta)} = b^{\overline{\la}}_{(\overline{\mu},\overline{\delta})}$.

    \item $\la / \mu$ is a  strip such that $|\la/\mu| \equiv \delta \: {\rm mod } \: 2$ if and only if $\overline{\la} / \overline{\mu}$ is a strip such that $|\overline{\la} / \overline{\mu}| \equiv \overline{\delta} \: {\rm mod } \: 2$.
    \end{enumerate}
        
\end{lemma}

\begin{proof}
    To prove (1), observe that upon restricting the determinant representation of $\O_k$ to the subgroup $\O_{k-1} \times \O_1$, we have $\mathbf{det}_k \cong \mathbf{det}_{k-1} \otimes \mathbf{det}_1$.
    Thus, starting from~\eqref{tensor det} and then restricting from $\O_k$ to $\O_{k-1} \times \O_1$, we have
    \begin{align*}
        E^{\overline{\la}}_k & \cong E^\la_k \otimes \mathbf{det}_k \\
        & \cong \left(\bigoplus_{\mu,\delta} b^\la_{(\mu,\delta)} E^\mu_{k - 1} \otimes E^\delta_1 \right) \otimes (\mathbf{det}_{k-1} \otimes \mathbf{det}_1) \\
        & \cong \bigoplus_{\mu,\delta} b^{\la}_{(\mu,\delta)} \left( E^\mu_{k-1} \otimes \mathbf{det}_{k-1} \right) \otimes \left( E^\delta_1 \otimes \mathbf{det}_1 \right) \\
        & \cong \bigoplus_{\mu,\delta} b^\la_{(\mu,\delta)}  E^{\overline{\mu}}_{k-1} \otimes E^{\overline{\delta}}_1 \\
        &= \bigoplus_{\mu, \delta} b^{\overline{\la}}_{(\overline{\mu}, \overline{\delta})} E^{\overline{\mu}}_{k-1} \otimes E^{\overline{\delta}}_1.
    \end{align*}

    To prove (2), suppose that $\la / \mu$ is a strip such that $|\la / \mu| \equiv \delta \: {\rm mod } \: 2$.
    Then $\ell(\la) - \ell(\mu)$ is either 0 or 1.
    Since $\ell(\overline{\la}) - \ell(\overline{\mu}) = (k-\ell(\la)) - (k-1 - \ell(\mu)) = 1 - (\ell(\la) - \ell(\mu))$, we have that $\ell(\overline{\la}) - \ell(\overline{\mu})$ is also either 0 or 1, and has the opposite parity as $\ell(\la) - \ell(\mu)$.
    Since associated partitions differ only in their first column, it follows that $\overline{\la} / \overline{\mu}$ is also a strip, and that $|\overline{\la} / \overline{\mu}|$ has the opposite parity as $|\la/\mu|$, so that $|\overline{\la} / \overline{\mu}| \equiv \overline{\delta} \: {\rm mod} \: 2$.
    Since $\la \mapsto \overline{\la}$ is an involution, the same argument also proves the converse.   
\end{proof}

\begin{proof}[Proof of Theorem~\ref{thm:K to M}]\

    (a) Suppose $K = \O_k$.
    We begin by finding a branching rule from $\O_{k}$ to $\O_{k-1} \times \O_1$.
    The proof then proceeds by iterating this branching rule to obtain the branching rule from $K$ to $M$.
    For this, let $\la \in \widehat{\O}_k$ with $\ell(\la) \leq k/2$.
    We apply the Howe duality setting of Theorem~\ref{thm:Howe duality} with $n = \lfloor k/2 \rfloor$.
    Let $\mu_{k-1} \in \widehat{\O}_{k-1}$ and $\delta_k \in \widehat{\O}_1 = \{0,1\}$.
    By seesaw reciprocity~\eqref{eq:seesaw}, we have
    \begin{equation}
    \label{first mult O}
    b^\la_{(\mu_{k-1}, \delta_k)} = \left[E^{\mu_{k-1}}_{k-1} \otimes E^{\delta_k}_1 : E^\la_{k} \right] =  \left[ \widetilde{E}^\la_{2n} : \widetilde{E}^{\mu_{k-1}}_{2n} \otimes \widetilde{E}^{\delta_k}_{2n} 
    \right]_{\textstyle{.}}
    \end{equation}
    Note that if $\ell(\mu_{k-1}) > n$, then $E^{\mu_{k-1}}_{k-1}$ and thus $\widetilde{E}^{\mu_{k-1}}_{2n}$ do not occur in this Howe duality setting.
    It follows from~\eqref{first mult O} and from the complete reducibility in seesaw reciprocity~\eqref{seesaw paradigm} that $b^\la_{(\mu_{k-1}, \delta_k)} = 0$ for such $\mu_{k-1}$.
    Therefore we may assume that $\ell(\mu_{k-1}) \leq k/2$.
    It then follows from the table in Lemma~\ref{lemma:stable} that we are in the stable range for both $\O_k$ and $\O_{k-1}$.
    Therefore, by Lemmas~\ref{lemma:stable} and~\ref{lemma:k=1}, we have the following isomorphisms of $\k' = \gl_n$-modules:
    \begin{align*}
        \widetilde{E}^\la_{2n} & \cong \C[{\rm SM}_n] \otimes F^{\la}_n,\\
        \widetilde{E}^{\mu_{k-1}}_{2n} & \cong \C[{\rm SM}_n] \otimes F^{\mu_{k-1}}_n,\\
        \widetilde{E}^{\delta_k}_{2n} & \cong \bigoplus_{\mathclap{\substack{m \in \mathbb{N}:\\
        m \equiv \delta_k \: {\rm mod} \: 2}}}
        F^{(m)}_n.
    \end{align*}
    Using these to rewrite~\eqref{first mult O}, we obtain
\begin{equation}
    \label{second mult O}
    b^\la_{(\mu_{k-1}, \delta_k)} = \left[\C[{\rm SM}_n] \otimes F^{\la}_n : \C[{\rm SM}_n] \otimes F^{\mu_{k-1}}_n \otimes \bigoplus_{\mathclap{\substack{m \in \mathbb{N}:\\
        m \equiv \delta_k \: {\rm mod} \: 2}}} F^{(m)}_n\right]_{\textstyle{.}}
\end{equation}
By the Pieri rule~\eqref{Pieri}, the $\k'$-module on the right-hand side of the colon in~\eqref{second mult O} becomes
\begin{align}
    & \phantom{=} \bigoplus_{\mathclap{\substack{m \in \mathbb{N}:\\
        m \equiv \delta_k \: {\rm mod} \: 2}}} \C[{\rm SM}_n] \otimes \left( F^{\mu_{k-1}}_n \otimes F^{(m)}_n \right) \nonumber \\
    & \cong \bigoplus_{\mathclap{\substack{m \in \mathbb{N}:\\
        m \equiv \delta_k \: {\rm mod} \: 2}}} \C[{\rm SM}_n] \otimes \Bigg(\bigoplus_{\substack{\ga: \\ \ga/\mu_{k-1} \text{ strip},\\
        |\ga/\mu_{k-1}| = m}}  \hspace{-3ex} F^{\ga}_n\Bigg) \nonumber \\
        & \cong \bigoplus_{\mathclap{\substack{\ga: \\ \ga / \mu_{k-1} \text{ strip}, \\ |\ga / \mu_{k-1} | \equiv \delta_{k} \: {\rm mod} \: 2}}} \C[{\rm SM}_n] \otimes F^\ga_n. \label{gamma O}
\end{align}
(We use the term \emph{strip} in the sense of Definition~\ref{def:strip}.) 
Define the following weight function on skew diagrams, which takes values in $\widehat{\O}_1 = \{0,1\}$:
\begin{equation}
    \label{wt O}
    {\rm wt}(\gamma / \mu) \coloneqq |\gamma / \mu| \: {\rm mod} \: 2.
\end{equation}
(We write this ``wt'' in Roman type to distinguish it from the $\widehat{M}$-valued function $\mathbf{wt}$ in Definition~\ref{def:KC-tableaux}.)
Substituting~\eqref{gamma O} into~\eqref{second mult O}, we obtain the following branching rule from $\O_k$ to $\O_{k-1} \times \O_{1}$:
\begin{equation}
    \label{Ok rule}
    b^\la_{(\mu_{k-1}, \delta_k)} = \begin{cases}
        1, & \text{$\la / \mu_{k-1}$ is a strip such that ${\rm wt}(\la / \mu_{k-1}) = \delta_k$},\\
        0 & \text{otherwise}.
    \end{cases} 
\end{equation}
Although we initially assumed that $\ell(\la) \leq k/2$, Lemma~\ref{lemma:Ok assoc rule} implies that the rule~\eqref{Ok rule} is valid for all $\la \in \widehat{\O}_k$, since $\ell(\la) > k/2$ implies that $\ell(\overline{\la}) < k/2$.
We thereby recover the following result of King~\cite{King75}*{(4.14)} and Proctor~\cite{Proctor}*{Prop.~10.1}:
\begin{equation}
    \label{O before iterating}
    E^\la_{k} \cong \bigoplus_{\delta_k \in \{0,1\}} \Bigg(\bigoplus_{\substack{\mu_{k-1} \in \widehat{\O}_{k-1} : \\
    \la / \mu_{k-1} \text{ strip,} \\ {\rm wt}(\la / \mu_{k-1}) = \delta_k}}   \hspace{-3ex} E^{\mu_{k-1}}_{k-1} \otimes E^{\delta_k}_1 \Bigg)_{\textstyle{.}}
\end{equation}
Starting with $\la = \mu_k$, we rewrite~\eqref{O before iterating} to obtain a branching rule from $\O_i$ to $\O_{i-1} \times \O_1$, for each $i = k, k-1, \ldots, 2$:
\begin{equation}
\label{iterate O}
    E^{\mu_{i}}_{i} \cong \bigoplus_{\delta_i \in \{0,1\}} \Bigg(\bigoplus_{\substack{\mu_{i-1} \in \widehat{\O}_{i-1} : \\
    \mu_i / \mu_{i-1} \text{ strip,} \\ {\rm wt}(\mu_i / \mu_{i-1}) = \delta_i}}   \hspace{-3ex} E^{\mu_{i-1}}_{i-1} \otimes E^{\delta_i}_1 \Bigg)_{\textstyle{.}}
\end{equation}
Iterating~\eqref{iterate O} to decompose $E^{\mu_i}_i$ for $i = k, k-1, \ldots, 2$, we obtain the following decomposition as a representation of $M = (\O_1)^k$:
\begin{equation}
\label{final O}
    E^\la_k \cong \bigoplus_{\bm{\delta} \in \{0,1\}^k} \left(\#\mathcal{M}^\la_{\bm{\delta}} \bigotimes_{i=1}^k E^{\delta_i}_1 \right)_{\textstyle{,}}
\end{equation}
where
\[
\mathcal{M}^\la_{\bm{\delta}} \coloneqq \left\{
\text{chains } (0 = \mu_0 \subseteq \mu_1 \subseteq \cdots \subseteq \mu_{k-1} \subseteq \mu_k = \la) : \; \parbox{4cm}{$\mu_i \in \widehat{\O}_i$, \\
$\mu_i / \mu_{i-1}$ is a strip,\\
and ${\rm wt}(\mu_i / \mu_{i-1}) = \delta_i$ \\ for all $1 \leq i \leq k$ }
\right\}_{\textstyle{.}}
\]

To complete the proof, we construct the following bijection between $\mathcal{M}^\la_{\bm{\delta}}$ and $\mathcal{T}(\O_k)^\la_{\bm{\delta}}$.
In one direction, let $\bmu \in \mathcal{M}^\la_{\bm{\delta}}$.
To construct a tableau $T \in \mathcal{T}(\O_k)^\la_{\bm{\delta}}$, start with the subtableau $T_1$ of shape $\mu_1$ filled with the entry $1$; note that since $\mu_1 \in \widehat{\O}_1$, this subtableau $T_1$ is either empty or contains a single box, and is (trivially) semistandard.
We construct the successive subtableaux $T_2 \subseteq \cdots \subseteq T_k = T$ in the same way: to construct $T_i$, fill every box in the skew diagram $\mu_i / \mu_{i-1}$ with the entry $i$, and append to $T_{i-1}$, thereby forming a new subtableau $T_i$ of shape $\mu_i$.
Note that $T_i$ remains semistandard, since $\mu_i / \mu_{i-1}$ is a strip and therefore each column contains at most one entry $i$.
Moreover, $T_i$ is an $\O_i$-tableau since $T_{i-1}$ is an $\O_{i-1}$-tableau and $\mu_i \in \widehat{\O}_i$.
Finally, we have $\mathbf{wt}(T_i) = (\delta_1, \ldots, \delta_i)$ because
\[
\#\{ \text{boxes in $T_i$ with entry $i$}\} \: {\rm mod} \: 2 = |\mu_i / \mu_{i-1}| \: {\rm mod} \: 2 = {\rm wt}(\mu_i / \mu_{i-1}) = \delta_i.
\]
Since $\mu_k = \la$, this process results in an $\O_k$-tableau $T_k = T \in \mathcal{T}(\O_k)^\la_{\bm{\delta}}$.
Conversely, in the other direction, let $T \in \mathcal{T}(\O_k)^\la_{\bm{\delta}}$.
To construct a chain $\bmu \in \mathcal{M}^\la_{\bm{\delta}}$, let $\mu_i$ be the shape of the subtableau of $T_i$ whose entries are $\leq i$, for all $1 \leq i \leq k$.
We leave it to the reader to check that these maps are mutually inverse.
We thus have $\#\mathcal{M}^\la_{\bm{\delta}} = \#\mathcal{T}(\O_k)^\la_{\bm{\delta}}$, and the result follows from~\eqref{final O}.

(b) Suppose $K = \GL_k$.
    As in part (a) above, we begin by finding a branching rule from $\GL_{k}$ to $\GL_{k-1} \times \GL_1$.
    Let $\la \in \widehat{\GL}_k$, and let $\mu_{k-1} \in \widehat{\GL}_{k-1}$.
    Let $\delta_k \in \widehat{\GL}_1 = \mathbb{Z}$.
    We now apply the Howe duality setting of Theorem~\ref{thm:Howe duality}, with $p = \ell(\la^+)$ and $q = \ell(\la^-)$.
    By seesaw reciprocity~\eqref{eq:seesaw}, we have
    \begin{equation}
    \label{first mult GL}
    b^\la_{(\mu_{k-1}, \delta_k)} = \left[F^{\mu_{k-1}}_{k-1} \otimes F^{\delta_k}_1 : F^\la_{k} \right] =  \left[ \widetilde{F}^\la_{p,q} : \widetilde{F}^{\mu_{k-1}}_{p,q} \otimes \widetilde{F}^{\delta_k}_{p,q} 
    \right]_{\textstyle{.}}
    \end{equation}
    Since $p+q = \ell(\la^+) + \ell(\la^-) \leq k$, it follows from the table in Lemma~\ref{lemma:stable} that we are in the stable range for both $\GL_k$ and $\GL_{k-1}$.
    It then follows from Lemmas~\ref{lemma:stable} and Lemma~\ref{lemma:k=1} that we have the following isomorphisms of $\k' = \gl_p \oplus \gl_q$-modules:
    \begin{align*}
        \widetilde{F}^\la_{p,q} & \cong \C[\M_{p,q}] \otimes F^{\la^+}_p \otimes F^{\la^-}_q,\\
        \widetilde{F}^{\mu_{k-1}}_{p,q} & \cong \C[\M_{p,q}] \otimes F^{\mu_{k-1}^+}_p \otimes F^{\mu_{k-1}^-}_q,\\
        \widetilde{F}^{\delta_k}_{p,q} & \cong \bigoplus_{\mathclap{\substack{m \in \mathbb{N}:\\
        m+\delta_k \geq 0}}}
        F^{(m+\delta_k)}_p \otimes F^{(m)}_q.
    \end{align*}
    By Theorem~\ref{thm:Howe duality}, $\widetilde{F}^{\mu_{k-1}}_{p,q} \neq 0$ if and only if $\ell(\mu^+_{k-1}) \leq p$ and $\ell(\mu^-_{k-1}) \leq q$. 
    Using the decompositions above to rewrite the right-hand side of~\eqref{first mult GL}, we obtain
\begin{equation}
    \label{second mult GL}
    b^\la_{(\mu_{k-1}, \delta_k)} = \left[\C[\M_{p,q}] \otimes F^{\la^+}_p \otimes F^{\la^-}_q : \C[\M_{p,q}] \otimes F^{\mu_{k-1}^+}_p \otimes F^{\mu_{k-1}^-}_q \otimes \bigoplus_{\mathclap{\substack{m \in \mathbb{N}:\\
        m+\delta_k \geq 0}}} F^{(m+\delta_k)}_p \otimes F^{(m)}_q\right]_{\textstyle{.}}
\end{equation}
Using the Pieri rule~\eqref{Pieri}, the $\k'$-module on the right-hand side of the colon in~\eqref{second mult GL} can be rewritten as
\begin{align}
    & \phantom{=} \bigoplus_{\mathclap{\substack{m \in \mathbb{N}:\\
        m+\delta_k \geq 0}}}  \C[\M_{p,q}] \otimes \left(F^{\mu_{k-1}^+}_p \otimes F^{(m+\delta_k)}_p\right) \otimes \left(F^{\mu_{k-1}^-}_q \otimes F^{(m)}_q\right) \nonumber \\
    & \cong \bigoplus_{\mathclap{\substack{m \in \mathbb{N}:\\
        m+\delta_k \geq 0}}}  \C[\M_{p,q}] \otimes \Bigg(\bigoplus_{\substack{\ga^+: \\ \ga^+ / \mu^+_{k-1} \text{ strip,} \\ |\ga^+ / \mu^+_{k-1}| = m + \delta_k}}  \hspace{-5ex} F^{\ga^+}_p\Bigg) \otimes \Bigg(\bigoplus_{\substack{\ga^-: \\ \ga^- / \mu^-_{k-1} \text{ strip,} \\ |\ga^- / \mu^-_{k-1}| = m}}  \hspace{-4ex} F^{\ga^-}_q\Bigg)_{\textstyle{.}}
        \label{GL rewrite}
\end{align}
Upon rearranging the direct sums, we rewrite~\eqref{GL rewrite} as
\begin{equation}
\label{gamma GL}
    \bigoplus_{\mathclap{\substack{\ga: \\ \ga^\pm / \mu^\pm_{k-1} \text{ strips,} \\ |\ga^+ / \mu^+_{k-1}| - |\ga^- / \mu^-_{k-1}| = \delta_k}}} \C[\M_{p,q}] \otimes F^{\ga^+}_p \otimes F^{\ga^-}_q.
\end{equation}

In general, for $\ga = (\ga^+, \ga^-)$ and $\mu = (\mu^+, \mu^-)$, we write $\mu \subseteq \gamma$ to express that $\mu^+ \subseteq \ga^+$ and $\mu^- \subseteq \ga^-$.
Define the following weight function on pairs $\gamma / \mu \coloneqq (\gamma^+/\mu^+, \gamma^- / \mu^-)$ of strips, taking values in $\widehat{\GL}_1 = \mathbb{Z}$:
\begin{equation*}
    \label{wt GL}
    {\rm wt}(\gamma / \mu) \coloneqq |\gamma^+ / \mu^+| - |\gamma^- / \mu^-|.
\end{equation*}
Substituting~\eqref{gamma GL} into~\eqref{second mult GL}, we obtain the following branching rule from $\GL_k$ to $\GL_{k-1} \times \GL_{1}$:
\[
    b^\la_{(\mu_{k-1}, \delta_k)} = 
    \begin{cases}
        1, & \la^\pm / \mu^\pm_{k-1} \text{ are strips and }{\rm wt}(\la / \mu_{k-1}) = \delta_k,\\
        0 & \text{otherwise}.
    \end{cases}
\]
This yields the following decomposition of $F^\la_k$:
\begin{equation}
    \label{GL before iterating}
    F^\la_{k} \cong \bigoplus_{\delta_k \in \mathbb{Z}} \Bigg(
    \bigoplus_{\substack{\mu_{k-1} \in \widehat{\GL}_{k-1}: \\ \la^\pm / \mu^\pm_{k-1} \text{ strips,} \\ {\rm wt}(\la / \mu_{k-1}) = \delta_k}} 
    F^{\mu_{k-1}}_{k-1} \otimes F^{\delta_k}_1 \Bigg)_{\textstyle{.}}
\end{equation}
Starting with $\la = \mu_k$ and iterating~\eqref{GL before iterating} to decompose $F^{\mu_i}_i$ for $i = k, k-1, \ldots, 2$, we obtain the following decomposition of $F^\la_k$ as a representation of $M = (\GL_1)^k$:
\begin{equation}
    \label{final GL}
    F^\la_{k} \cong \bigoplus_{\bm{\delta} \in \mathbb{Z}^k} \left( \#\mathcal{M}^\la_{\bm{\delta}} \bigotimes_{i=1}^k F^{\delta_i}_1\right)_{\textstyle{,}}
\end{equation}
where
\[
\mathcal{M}^\la_{\bm{\delta}} \coloneqq \left\{
\text{chains } (0 = \mu_0 \subseteq \mu_1 \subseteq \cdots \subseteq \mu_{k-1} \subseteq \mu_k = \la) : \;\; \parbox{4cm}{$\mu_i \in \widehat{\GL}_i$, \\
$\mu_i^\pm / \mu_{i-1}^\pm$ are strips,\\
and ${\rm wt}(\mu_i / \mu_{i-1}) = \delta_i$ \\
for all $1 \leq i \leq k$}
\right\}_{\textstyle{.}}
\]

There is a natural bijection $\mathcal{M}^\la_{\bm{\delta}} \longrightarrow \mathcal{T}(\GL_k)^\la_{\bm{\delta}}$, given in a similar way as the proof of part (a) above.
We therefore have $\#\mathcal{M}^\la_{\bm{\delta}} = \#\mathcal{T}(\GL_k)^\la_{\bm{\delta}}$, and the result follows from~\eqref{final GL}.

    (c) Suppose $K = \Sp_{2k}$.
    As in parts (a) and (b) above, we begin by finding a branching rule from $\Sp_{2k}$ to $\Sp_{2(k-1)} \times \Sp_2$.
    Let $\la \in \widehat{\Sp}_{2k}$, and let $\mu \in \widehat{\Sp}_{2(k-1)}$.
    Let $\delta_k \in \widehat{\Sp}_2 = \mathbb{N}$.
    We now apply the Howe duality setting of Theorem~\ref{thm:Howe duality} with $n=k$.
    By seesaw reciprocity~\eqref{eq:seesaw}, we have
    \begin{equation}
    \label{first mult Sp}
    b^\la_{(\mu_{k-1}, \delta_k)} = \left[V^{\mu_{k-1}}_{2(k-1)} \otimes V^{\delta_k}_2 : V^\la_{2k} \right] =  \left[ \widetilde{V}^\la_{2n} : \widetilde{V}^{\mu_{k-1}}_{2n} \otimes \widetilde{V}^{\delta_k}_{2n} \right].
    \end{equation}
    Since $n=k$, it follows from the table in Lemma~\ref{lemma:stable} that we are in the stable range for both $\Sp_{2k}$ and $\Sp_{2(k-1)}$.
    It then follows from Lemmas~\ref{lemma:stable} and~\ref{lemma:k=1} that we have the following isomorphisms of $\k' = \gl_n$-modules:
    \begin{align*}
        \widetilde{V}^\la_{2n} & \cong \C[{\rm AM}_n] \otimes F^\la_n,\\
        \widetilde{V}^{\mu_{k-1}}_{2n} & \cong \C[{\rm AM}_n] \otimes F^{\mu_{k-1}}_n,\\
        \widetilde{V}^{\delta_k}_{2n} & \cong \bigoplus_{m \in \mathbb{N}} F^{(m+\delta_k,m)}_n.
    \end{align*}
    Using these to rewrite~\eqref{first mult Sp}, we obtain
\begin{equation}
    \label{second mult Sp}
    b^\la_{(\mu_{k-1}, \delta_k)} = \left[ \C[{\rm AM}_n] \otimes F^\la_n : \C[{\rm AM}_n] \otimes F^{\mu_{k-1}}_n \otimes \bigoplus_{m \in \mathbb{N}}F^{(m+\delta_k,m)}_n \right]_{\textstyle{,}}
\end{equation}
where the $\k'$-module on the right-hand side of the colon can be rewritten as
\begin{align}
    & \phantom{=} \bigoplus_{m\in \mathbb{N}} \C[{\rm AM}_n] \otimes \left(F^{\mu_{k-1}}_n \otimes F^{(m+\delta_k,m)}_n\right) \nonumber \\
    & \cong \bigoplus_{m\in \mathbb{N}} \C[{\rm AM}_n] \otimes \bigoplus_\ga c^\ga_{\mu_{k-1}, (m+\delta_k,m)} \: F^\ga_n. \label{rewrite Sp over m}
\end{align}
The Littlewood--Richardson rule~\eqref{LR rule} states that
\begin{equation}
\label{Sp inductive}
    c^\ga_{\mu_{k-1}, (m+\delta_k,m)} = \# \left\{ \text{LR tableaux of shape $\ga / \mu_{k-1}$ and content $(m+\delta_k, m)$} \right\}.
\end{equation}
By~\eqref{LR double} this coefficient is nonzero only if $\gamma / \mu_{k-1}$ is a double strip, in the sense of Definition~\ref{def:strip}.
If $S$ is an LR tableau with content $(\ell, m)$, then define the following weight function with values in $\widehat{\Sp}_2 = \mathbb{N}$:
\begin{equation}
    \label{wt Sp}
    {\rm wt}(S) \coloneqq \ell - m.
\end{equation}
Equivalently, ${\rm wt}(S)$ equals the number of 1's in $S$ minus the number of 2's in $S$.
For $\ga / \mu$ a double strip and $\delta \in \mathbb{N}$, define the following collection of LR tableaux:
\begin{equation}
\label{L}
    \mathcal{S}^{\gamma / \mu}_{\delta} \coloneqq \{ \text{LR tableaux $S$ : $S$ is filled with $1$'s and $2$'s, $S$ has  shape $\gamma / \mu$, and  ${\rm wt}(S) = \delta$}\}.
\end{equation}
Using~\eqref{Sp inductive}--\eqref{L}, we rewrite~\eqref{rewrite Sp over m} as
\begin{equation}
\label{gamma Sp}
    \bigoplus_{\mathclap{\substack{\gamma: \\ \gamma / \mu_{k-1} \text{ dbl. strip}}}} \#\mathcal{S}^{\gamma/\mu_{k-1}}_{\delta_k} (\C[{\rm AM}_n] \otimes F^\gamma_n).
\end{equation}
Substituting~\eqref{gamma Sp} into~\eqref{second mult Sp}, we obtain the following branching rule from $\Sp_{2k}$ to $\Sp_{2(k-1)} \times \Sp_2$:
\[
b^\la_{(\mu_{k-1}, \delta_k)} = \#\mathcal{S}^{\la / \mu_{k-1}}_{\delta_k}.
\]
This gives the following decomposition of $V^\la_{2k}$:
\begin{equation}
    \label{Sp before iterating}
    V^\la_{2k} \cong 
    \bigoplus_{\delta_k \in \mathbb{N}}
    \Bigg(     \bigoplus_{\substack{\mu_{k-1} \in \widehat{\Sp}_{2(k-1)}: \\ \la / \mu_{k-1} \text{ dbl. strip}}}
    \hspace{-4ex} \#\mathcal{S}^{\la/\mu_{k-1}}_{\delta_k} 
    (V^{\mu_{k-1}}_{2(k-1)} \otimes V^{\delta_k}_2) \Bigg)_{\textstyle{.}}
\end{equation}
 (This is a more explicit version of the formula due to King~\cite{King75}*{equation (4.15)} and Proctor~\cite{Proctor}*{Prop.~10.3}.)
Starting with $\la = \mu_k$ and iterating~\eqref{Sp before iterating} to decompose $V^{\mu_i}_{2i}$ for $i = k, k-1, \ldots, 2$, we obtain the following decomposition as a representation of $M = (\Sp_2)^k$:
\begin{equation}
    \label{final Sp}
    V^\la_{2k} \cong \bigoplus_{\bm{\delta} \in \mathbb{N}^k} \left( \#\mathcal{M}^\la_{\bm{\delta}} \bigotimes_{i=1}^k V^{\delta_i}_2 \right)_{\textstyle{,}}
\end{equation}
where $\mathcal{M}^{\la}_{\bm{\delta}}$ is the following set consisting of $k$-tuples of LR tableaux:
\begin{equation}
    \label{big M sp}
    \mathcal{M}^\la_{\bm{\delta}} \coloneqq \left\{
(S_1, \ldots, S_k) : \; \parbox{10cm}{there exists a chain $(0 = \mu_0 \subseteq \mu_1 \subseteq \cdots \subseteq \mu_{k-1} \subseteq \mu_k = \la)$\\
such that $\mu_i \in \widehat{\Sp}_{2i}$,\\
$\mu_i / \mu_{i-1}$ is a double strip,\\
and $S_i \in \mathcal{S}^{\mu_i / \mu_{i-1}}_{\delta_i}$ for all $1 \leq i \leq k$}
\right\}_{\textstyle{.}}
\end{equation}

To complete the proof, we construct a bijection $\mathcal{M}^\la_{\bm{\delta}} \longrightarrow \mathcal{T}(\Sp_{2k})^\la_{\bm{\delta}}$ in the same way as in parts (a) and (b) above.
In one direction, let $(S_1, \ldots, S_k) \in \mathcal{M}^\la_{\bm{\delta}}$, and let $\bmu$ be the chain such that each $\mu_i / \mu_{i-1}$ is the shape of $S_i$.
To construct a tableau $T \in \mathcal{T}(\Sp_{2k})^\la_{\bm{\delta}}$, start with the subtableau $T_1 = S_1$.
Note that because $\mu_1 \in \widehat{\Sp}_2$ and $S_1 \in \mathcal{S}^{\mu_1}_{\delta_1}$, this subtableau $T_1$ is a single row consisting of exactly $\delta_1$ boxes, all containing the entry 1.
Therefore we have $T_1 \in \mathcal{T}(\Sp_2)^{\mu_1}_{\delta_1}$, where necessarily $\mu_1 = (\delta_1)$.
We construct the successive subtableaux $T_2 \subseteq \cdots \subseteq T_k = T$ in the same way, as follows.
Suppose we have constructed $T_i \in \mathcal{T}(\Sp_{2(i-1)})^{\mu_i-1}_{(\delta_1, \ldots, \delta_{i-1})}$ an $\Sp_{2(i-1)}$-ballot tableau of shape $\mu_{i-1}$ and weight $(\delta_1, \ldots, \delta_{i-1})$.
To construct $T_i$ from $T_{i-1}$, take $S_i$, replace the entries 1 and 2 with $i$ and $\bar{\imath}$, respectively, and append to $T_{i-1}$.
Note that $T_i$ belongs to $\mathcal{T}(\Sp_{2i})^{\mu_i}_{(\delta_1, \ldots, \delta_i)}$: in particular, it is an $\Sp_{2i}$-ballot tableau because $\mu_i \in \widehat{\Sp}_{2i}$, $T_{i-1}$ is an $\Sp_{2(i-1)}$-ballot tableau, and because $S_i$ is an LR tableau.
We have $\mathbf{wt}(T_i) = (\delta_1, \ldots, \delta_i)$ because
\begin{align*}
& \phantom{=.} \#\{\text{boxes in $T_i$ with entry $i$}\} - \#\{\text{boxes in $T_i$ with entry $\bar{\imath}$}\} \\
& = \#\{ \text{boxes in $S_i$ with entry $1$} \} - \#\{ \text{boxes in $S_i$ with entry $2$} \} = {\rm wt}(S_i) = \delta_i,
\end{align*}
where the second equality follows from~\eqref{wt Sp}, and the third equality follows from the definitions of $\mathcal{S}^{\mu_i / \mu_{i-1}}_{\delta_i}$ and $\mathcal{M}^\la_{\bm{\delta}}$ in~\eqref{L} and~\eqref{big M sp}.
Since $\mu_k = \lambda$, this process results in a tableau $T \in \mathcal{T}(\Sp_{2k})^\la_{\bm{\delta}}$.
Conversely, in the other direction, let $T \in \mathcal{T}(\Sp_{2k})^\la_{\bm{\delta}}$.
To construct a tuple $(S_1, \ldots, S_k) \in \mathcal{M}^\la_{\bm{\delta}}$, let $S_i$ consist of the boxes in $T$ with entries $i$ or $\bar{\imath}$; then replace each $i$ and $\bar{\imath}$ by $1$ and $2$, respectively.
Because $T \in \mathcal{T}(\Sp_{2k})^\la_{\bm{\delta}}$, it follows that each $S_i \in \mathcal{S}^{\mu_i / \mu_{i-1}}_{\delta_i}$, and therefore $(S_1, \ldots, S_k) \in \mathcal{M}^\la_{\bm{\delta}}$.
We leave it to the reader to check that these maps are mutually inverse.
We thus have $\# \mathcal{M}^\la_{\bm{\delta}} = \#\mathcal{T}(\Sp_{2k})^\la_{\bm{\delta}}$, and the result follows from~\eqref{final Sp}.
\end{proof}

\section{Special cases of the branching rule}
\label{sec:examples}

As an application of Theorem~\ref{thm:K to M}, we can easily determine the multiplicities $b^\la_{\bm{\delta}}$ in the special case where $\la$ (or each of $\la^+$ and $\la^-$, in the case where $K = \GL_k$) is a single row or a single column.
We use the shorthand $(1^a) \coloneqq (1, \ldots, 1)$, where $1$ is repeated $a$ times.
To express the branching rules below, for $\bm{\delta} = (\delta_1, \ldots, \delta_k) \in \widehat{M}$, we define
\[
|\bm{\delta}| \coloneqq \sum_{i=1}^k |\delta_i|,
\]
which is a nonnegative integer by~\eqref{M hat}.

We first recall the notion of a multiset.
A \emph{multiset} is a collection of elements in which the elements are allowed to occur multiple times.
For example, $[1,1,2]$ and $[1,2,2,2]$ are distinct multisets that are formed from the same finite set $\{1,2\}$.
The number of times a given element occurs in a multiset is referred to as the \emph{multiplicity} of the element.
The cardinality of the multiset is the sum of the multiplicities of its elements, so that $|[1,1,2]| = 3$ and $|[1,2,2,2]| = 4$.
We will use the following notation for multisets constructed from a finite subset $A = \{a_1, \ldots, a_k\}$.
We denote the multiset formed from $A$ where the element $a_i$ has multiplicity $m_i \geq 0$ by $\{a_1^{m_1}, a_2^{m_2}, \ldots, a_k^{m_k}\}$.
The cardinality of such a multiset is then $\sum_{i=1}^k m_i$.
It is well known~\cite{Stanley}*{p.~26} that the number of multisets of cardinality $n$ that can be constructed with elements taken from a finite set of cardinality $k$ is given by
\begin{equation}
    \label{mset binomial}
    \left(\!\!\binom{k}{n}\!\!\right) \coloneqq \binom{k+n-1}{n} = \binom{k+n-1}{k-1}_{\textstyle{.}}
\end{equation}

We point out that in the following corollary, part (a) gives the branching rule for the spherical harmonics, since the irreducible representation $E^{(a)}_k$ is realized by the space of homogeneous $\O_k$-spherical harmonic polynomials of degree $a$ in $\C[x_1, \ldots, x_k]$.
Moreover, part (b) includes the weight multiplicity formula for the nontrivial component of the adjoint representation of $\GL_k$, which corresponds to the special case $b=c=1$.

\begin{cor}[$\la$ or $\la^\pm$ has one row]
    \label{cor:one-row}
    Let $\bm{\delta} \in \widehat{M}$.

    \begin{enumerate}[label=\textup{(\alph*)}]
        \item Let $K = \O_k$, with $a \in \mathbb{N}$.
        We have 
        \[
        b^{(a)}_{\bm{\delta}} = \begin{cases}
            \displaystyle\binom{\frac{a - |\bm{\delta}|}{2} + k-2}{k-2}, & a - |\bm{\delta}| \in 2\mathbb{N},\\[2ex]
            0 & \textup{otherwise}.
        \end{cases}
        \]

        \item Let $K = \GL_k$, with $b,c \in \mathbb{N}$ and $a \coloneqq b+c$.
        We have 
        \[
        b^{(b, 0, \ldots, 0, -c)}_{\bm{\delta}} = \begin{cases}
            \displaystyle\binom{\frac{a- |\bm{\delta}|}{2} + k-2}{k-2}, & b-c = \sum_{i=1}^k \delta_i,\\[2ex]
            0 & \textup{otherwise}.
        \end{cases}
        \]

        \item Let $K = \Sp_{2k}$, with $a \in \mathbb{N}$.
        We have
        \[
        b^{(a)}_{\bm{\delta}} = \begin{cases}
            1, & a = |\bm{\delta}|,\\
            0 & \textup{otherwise}.
        \end{cases}
        \]
    \end{enumerate}
    
\end{cor}

\begin{proof}\

(a) By Theorem~\ref{thm:K to M}, $b^{(a)}_{\bm{\delta}} = \#\mathcal{T}(\O_k)^{(a)}_{\bm{\delta}}$, where $\mathcal{T}(\O_k)^{(a)}_{\bm{\delta}}$ is the set of $\O_k$-tableaux of shape $(a)$ such that $w_i \equiv \delta_i \: {\rm mod} \: 2$, where $w_i$ denotes the number of occurrences of the entry $i$.
Note that $a = \sum_{i=1}^k w_i$, so that $a - |\bm{\delta}|$ is a nonnegative even integer.
We construct a bijection between $\mathcal{T}(\O_k)^{(a)}_{\bm{\delta}}$ and the collection of all multisets of cardinality $(a - |\bm{\delta}|)/2$ that can be formed from the set $\{2, \ldots, k\}$.
In one direction, let $T \in \mathcal{T}(\O_k)^{(a)}_{\bm{\delta}}$.
By Definition~\ref{def:KC-tableaux}, we have $w_1 \leq 1$.
Since $T \in \mathcal{T}(\O_k)^{(a)}_{\bm{\delta}}$, $w_1 = 1$ if and only if $\delta_i = 1$; for $i \in \{2, \ldots, k\}$,  $w_i$ is even if $\delta_i = 0$ and $w_i$ is odd if $\delta_i = 1$.
(See~\eqref{T(KC)}.)
Associate to $T$ the multiset $\mathcal{A}_T \coloneqq \{2^{\ell_2}, 3^{\ell_3}, \ldots, k^{\ell_k}\}$, where $\ell_i = w_i / 2$ if $\delta_i = 0$ and $\ell_i = (w_i - 1)/2$ if $\delta_i = 1$.
Note that $\mathcal{A}_T$ is a multiset of cardinality $\sum_{i=2}^k \ell_i = (a - |\bm{\delta}|)/2$ consisting of elements taken from the set $\{2, \ldots, k\}$.
Conversely, given a multiset $\{2^{m_2}, 3^{m_3}, \ldots, k^{m_k}\}$ with each $m_i \geq 0$ and with $\sum_{i=2}^k m_i = (a - |\bm{\delta}|)/2$, we can construct a tableau $T \in \mathcal{T}(\O_k)^{(a)}_{\bm{\delta}}$ by filling boxes from left to right with the numbers $1, \ldots, k$ in weakly increasing order as follows.
If $\delta_1 = 1$, then start the tableau with the entry $1$, and if $\delta_1 = 0$, then do not start with a $1$.
Next, for each $i = 2, \ldots, k$, fill $2m_i + \delta_i$ many boxes with the entry $i$.
In this way, we produce a one-row $\O_k$-tableau $T$ with exactly $|\bm{\delta}| + 2\sum_{i=2}^k m_i = a$ many boxes, such that $\mathbf{wt}(T) \equiv \bm{\delta} \: {\rm mod} \: 2$.
We leave it to the reader to check that these maps are mutually inverse.
The result then follows from~\eqref{mset binomial}.

(b) By Theorem~\ref{thm:K to M}, we have $b^{(b,0,\ldots, 0, -c)}_{\bm{\delta}} = \#\mathcal{T}(\GL_k)^{((b),(c))}_{\bm{\delta}}$, where $\mathcal{T}(\GL_k)^{((b),(c))}_{\bm{\delta}}$ is the set of $\GL_k$-tableaux $(T^+, T^-)$ of shape $((b), (c))$ such that $w_i^+ - w_i^- = \delta_i$, where $w_i^\pm$ denotes the number of occurrences of the entry $i$ in $T^\pm$.
Note that $b-c = \sum_{i=1}^k \delta_i$.
We construct a bijection between $\mathcal{T}(\GL_k)^{((b),(c))}_{\bm{\delta}}$ and the collection of all multisets of cardinality $(b+c-|\bm{\delta}|)/2$ that can be formed from the set $\{2, \ldots, k\}$.
In one direction, let $T = (T^+, T^-) \in \mathcal{T}(\GL_k)^{((b),(c))}_{\bm{\delta}}$.
By Definition~\ref{def:KC-tableaux}, at least one of $w_1^+$ and $w_1^-$ equals $0$.
Further, since $T \in \mathcal{T}(\GL_k)^{((b),(c))}_{\bm{\delta}}$, we have $w_1^+ = \delta_1$ if $\delta_1 \geq 0$, and $w_1^- = -\delta_1$ if $\delta_1 \leq 0$.
(See~\eqref{T(KC)}.)
Associate to $T$ the multiset $\mathcal{A}_T \coloneqq \{2^{\ell_2}, 3^{\ell_3}, \ldots, k^{\ell_k}\}$, where $\ell_i = \min\{w^+_i, w^-_i\} = (w^+_i + w^-_i - |\delta_i|)/2$.
Note that $\mathcal{A}_T$ is a multiset of cardinality $\sum_{i=2}^k \ell_i = (b+c - |\bm{\delta}|)/2$ consisting of elements taken from the set $\{2, \ldots, k\}$.
Conversely, given a multiset $\{2^{m_2}, 3^{m_3}, \ldots, k^{m_k}\}$ with each $m_i \geq 0$ and with $\sum_{i=2}^k m_i = (b+c - |\bm{\delta}|)/2$, we can construct a pair $(T^+, T^-) \in \mathcal{T}(\GL_k)^{((b),(c))}_{\bm{\delta}}$ by filling boxes as follows.
Fill the first $|\delta_1|$ many boxes with the entry $1$ in the tableau $T^{{\rm sgn}(\delta_1)}$, where ${\rm sgn}(\delta_1)$ is the sign of $\delta_1$ (omit this step if $\delta_1 = 0$).
Next, for each $i = 2, \ldots, k$, fill $m_i$ many boxes with the entry $i$ in both $T^+$ and $T^-$; then fill $|\delta_i|$ many boxes with the entry $i$ in the tableau $T^{{\rm sgn}(\delta_i)}$ (or omit this step if $\delta_i = 0$).
In this way we produce a $\GL_k$-tableau $(T^+, T^-)$ of weight $\bm{\delta}$, where $T^+$ and $T^-$ have one row each, with a total of $|\bm{\delta}| + 2\sum_{i=2}^k m_i = b+c$ many boxes, such that $|T^+| - |T^-| = \sum_{i=1}^k \delta_i = b-c$.
It follows that $(T^+, T^-)$ has shape $((b),(c))$.
We leave it to the reader to check that these maps are mutually inverse.
The result then follows from~\eqref{mset binomial}.

(c) By Theorem~\ref{thm:K to M}, the multiplicity $b^{(a)}_{\bm{\delta}}$ is the number of $\Sp_{2k}$-ballot tableaux of shape $(a)$, in which each $\delta_i$ equals the difference between the number of entries $i$ and $\bar{\imath}$ (in that order).
But because we consider only $\Sp_{2k}$-\emph{ballot} tableaux (see Definition~\ref{def:KC-tableaux}), and because each $\delta_i \geq 0$ (by~\eqref{M hat}), a one-row tableau cannot contain any barred entries $\bar{\imath}$.
Therefore the weight $\bm{\delta}$ determines a unique valid filling of a one-row $\Sp_{2k}$-ballot tableau if and only if $|\bm{\delta}|$ equals the number of boxes in the tableau (and otherwise no filling is possible). \qedhere
\end{proof}

In the proof of parts (b) and (c) of the following corollary, we make use of \emph{ballot sequences}, which we define as follows.
Let $S$ be a finite set $\{s_1, \ldots, s_\ell\}$ of positive integers such that $s_1 < \cdots < s_\ell$, and let $x$ and $y$ be nonnegative integers such that $x + y = \ell$.
An \emph{$(x,y)$-ballot sequence over $S$} is a sequence $(s_1, \ldots, s_\ell)$ where exactly $y$ many terms are \emph{marked} with a star, such that as the sequence is read from left to right, the number of marked terms never exceeds the number of unmarked terms.
For example, if $S = \{2,4,5,7,9\}$, then $(2, 4^*, 5, 7, 9^*)$ is a $(3,2)$-ballot sequence over $S$.
Note that $(2, 4^*, 5^*, 7, 9)$ is not a ballot sequence, because within the first three terms, two are marked.
We observe that the number of $(x,y)$-ballot sequences over $S$ equals the number of Dyck paths from $(0,0)$ to $(x,y)$, which is well known~\cite{Carlitz}*{equation (2.12)} to be given by the \emph{ballot number}
\begin{equation}
    \label{lattice path formula}
    \frac{x-y+1}{x+y+1}\binom{x+y+1}{y}_{\textstyle{.}}
\end{equation}

\begin{cor}[$\la$ or $\la^\pm$ has one column]
\label{cor:one-column}
    Let $\bm{\delta} \in \widehat{M}$.

    \begin{enumerate}[label=\textup{(\alph*)}]
        \item If $K = \O_k$, with $0 \leq a \leq k$, then we have
        \[
        b^{(1^a)}_{\bm{\delta}} = \begin{cases}
            1, & a = |\bm{\delta}|,\\
            0 & \textup{otherwise}.
        \end{cases}
        \]

        \item If $K = \GL_k$, with $b,c \in \mathbb{N}$ such that $a \coloneq b+c \leq k$, then we have
        \[
        b^{(1^b, 0, \ldots, 0, -1^c)}_{\bm{\delta}} = 
        \begin{cases}
            \displaystyle\frac{k-a+1}{k-|\bm{\delta}|+ 1} \binom{k-|\bm{\delta}| + 1}{\frac{a-|\bm{\delta}|}{2}}, & \bm{\delta} \in \{-1,0,1\}^k \textup{ and } b-c = \sum_{i=1}^k \delta_i,\\[2ex]
            0 & \textup{otherwise}.
            \end{cases}
            \]
            
        \item If $K = \Sp_{2k}$, with $0 \leq a \leq k$, then we have
        \[
        b^{(1^a)}_{\bm{\delta}} = 
        \begin{cases}
            \displaystyle\frac{k-a+1}{k-|\bm{\delta}|+ 1} \binom{k-|\bm{\delta}| + 1}{\frac{a-|\bm{\delta}|}{2}}, & \bm{\delta} \in \{0,1\}^k \textup{ and } a - |\bm{\delta}| \in 2\mathbb{N},\\[2ex]
            0 & \textup{otherwise}.
            \end{cases}
            \]
        
    \end{enumerate}

\end{cor}

\begin{proof}\

(a) By Theorem~\ref{thm:K to M}, $b^{(1^a)}_{\bm{\delta}} = \#\mathcal{T}(\O_k)^{(1^a)}_{\bm{\delta}}$, where $\mathcal{T}(\O_k)^{(1^a)}_{\bm{\delta}}$ is the set of $\O_k$-tableaux of shape $(1^a)$ such that each entry $i$ appears an even (resp., odd) number of times if $\delta_i = 0$ (resp., 1).
Since the shape $(1^a)$ is a single column, each entry appears at most once.
Therefore, the weight $\bm{\delta}$ determines a unique filling if and only if $a = |\bm{\delta}|$, and otherwise no filling is possible.

(b) By Theorem~\ref{thm:K to M}, $b^{(1^a)}_{\bm{\delta}} = \#\mathcal{T}(\GL_k)^{(1^b, 0, \ldots, 0, -1^c)}_{\bm{\delta}}$, where $\mathcal{T}(\GL_k)^{(1^b, 0, \ldots, 0, -1^c)}_{\bm{\delta}}$ is the set of $\GL_{k}$-tableaux $(T^+, T^-)$ of shape $((1^b), (1^c))$ such that the number of $i$'s in $T^+$ minus the number of $i$'s in $T^-$ equals $\delta_i$.
Since each of $T^+$ and $T^-$ is semistandard with a single column, no entry occurs more than once in each.
Therefore,  for $\mathcal{T}(\GL_k)^{(1^b, 0, \ldots, 0, -1^c)}_{\bm{\delta}} \neq \varnothing$ we require that $\bm{\delta} \in \{-1,0,1\}^k$.
It follows that $b-c = \sum_{i=1}^k \delta_i$.
Let
\begin{equation}
    \label{Z}
    Z(\bm{\delta}) \coloneqq \{i: \delta_i = 0\},
\end{equation}
so that
\begin{equation}
    \label{size Z}
    \# Z(\bm{\delta}) = k - |\bm{\delta}|.
\end{equation}
We construct a bijection between $\mathcal{T}(\GL_k)^{(1^b, 0, \ldots, 0, -1^c)}_{\bm{\delta}}$ and the collection of $(\frac{2k- a -|\bm{\delta}|}{2}, \: \frac{a-|\bm{\delta}|}{2})$-ballot sequences over $Z(\bm{\delta})$, where ${a = b+c}$.

In one direction, let $T = (T^+, T^-) \in \mathcal{T}(\GL_k)^{(1^b, 0, \ldots, 0, -1^c)}_{\bm{\delta}}$.
Associate to $T$ the increasing sequence $B_T$ consisting of the elements of $Z(\bm{\delta})$, where each term $i \in Z(\bm{\delta})$ is marked with a star if and only if the entry $i$ occurs in both $T^+$ and $T^-$.
We claim that $B_T$ is a $(\frac{2k-a-|\bm{\delta}|}{2}, \frac{a-|\bm{\delta}|}{2})$-ballot sequence.
We first show that $B_T$ has the specified number of marked and unmarked terms.
Indeed, if $i \notin Z(\bm{\delta})$, then it follows from definitions that the entry $i$ occurs in $T^+$ or $T^-$ but not in both.
The number of such indices $i$ is clearly $|\bm{\delta}|$.
Since $|T^+| + |T^-| = b+c = a$, it follows that the number of marked terms in $B_T$ is $\frac{a-|\bm{\delta}|}{2}$.
By~\eqref{size Z}, the number of unmarked terms in $B_T$ is therefore $\# Z(\bm{\delta}) - \frac{a - |\bm{\delta}|}{2} = k - |\bm{\delta}| - \frac{a - |\bm{\delta}|}{2} = \frac{2k-a-|\bm{\delta}|}{2}$.
It remains to show that $B_T$ is indeed a ballot sequence.
Let $(B_T)_{\leq i}$ denote the set of terms in $B_T$ which are at most $i$, and let $(B_T)^*_{\leq i}$ denote the set of marked terms in $B_T$ which are at most $i$.
Note that
\begin{equation}
    \label{i equation}
    \#\{\text{pos. integers $\leq i$ not in $Z(\bm{\delta})$} \} + \#(B_T)_{\leq i} = i.
\end{equation}
By construction, since $T$ is a $\GL_{k}$-tableau,
\begin{align}
    \#\{\text{boxes in $T^+$ or $T^-$ with entry $\leq i$}\} & = \#\{\text{pos. integers $\leq i$ not in $Z(\bm{\delta})$} \} + 2 \cdot \#(B_T)^*_{\leq i} \nonumber \\
    & \leq i. \label{i inequality}
\end{align}
It follows from~\eqref{i equation} that $2 \cdot \#(B_T)^*_{\leq i} \leq \#(B_T)_{\leq i}$ for all $1 \leq i \leq k$, and therefore $B_T$ is a ballot sequence.

Conversely, given a $(\frac{2k - a -|\bm{\delta}|}{2}, \: \frac{a-|\bm{\delta}|}{2})$-ballot sequence $B$ over $Z(\bm{\delta})$, we can construct a tableau in $\mathcal{T}(\GL_k)^{(1^b, 0, \ldots, 0, -1^c)}_{\bm{\delta}}$ as follows.
Starting with $i=1$ and proceeding to $i=k$, construct a tableau $T = (T^+, T^-)$ from the Young diagrams for $(1^b)$ and $(1^c)$, respectively, according to the following rules.
If $\delta_i = +1$, then place the entry $i$ in the next empty box of $T^+$.
If $\delta_i = -1$, then place the entry $i$ in the next empty box of $T^-$.
If $i^*$ is a marked term in $B$, then place $i$ in the next empty box of both $T^+$ and $T^-$.
If $i$ satisfies none of these conditions, then proceed to $i+1$.
Since $B$ is a ballot sequence, $2 \cdot \#(B)^*_{\leq i} \leq (B)_{\leq i}$ for all $1 \leq i \leq k$, and it follows from definitions that~\eqref{i equation} holds when $B_T$ is replaced by $B$.
Replacing $B_T$ with $B$ in~\eqref{i inequality}, we then see that $T$ is a $\GL_k$-tableau.
The fact that $\mathbf{wt}(T) = \bm{\delta}$ is clear from the construction of $T$.
We leave it to the reader to check that these maps are mutually inverse.
The result then follows from~\eqref{lattice path formula}.

(c) By Theorem~\ref{thm:K to M}, $b^{(1^a)}_{\bm{\delta}} = \#\mathcal{T}(\Sp_{2k})^{(1^a)}_{\bm{\delta}}$, where $\mathcal{T}(\Sp_{2k})^{(1^a)}_{\bm{\delta}}$ is the set of $\Sp_{2k}$-ballot tableaux of shape $(1^a)$ such that the number of $i$'s minus the number of $\bar{\imath}$'s equals $\delta_i$.
Since any such tableau is semistandard with a single column, no entry occurs more than once.
Therefore,  for $\mathcal{T}(\Sp_{2k})^{(1^a)}_{\bm{\delta}} \neq \varnothing$ we require that $\bm{\delta} \in \{0,1\}^k$.
Let $Z(\bm{\delta})$ be as defined in~\eqref{Z}.
We construct a bijection between $\mathcal{T}(\Sp_{2k})^{(1^a)}_{\bm{\delta}}$ and the collection of $(\frac{2k- a -|\bm{\delta}|}{2}, \: \frac{a-|\bm{\delta}|}{2})$-ballot sequences over $Z(\bm{\delta})$.

The proof proceeds analogously to the proof of part (b).
In this case, $i^*$ is marked in $B_T$ if and only if both $i$ and $\bar{\imath}$ occur in $T$.
Since $T$ is an $\Sp_{2k}$-tableau, the analogue of~\eqref{i inequality} in this case yields that $B_T$ is a ballot sequence.
Conversely, starting with a ballot sequence $B$, we construct a tableau $T$ as in the proof of part (b);
we place an $i$ in $T$ for each $i \notin Z(\bm{\delta})$, and we place an $i$ and $\bar{\imath}$ (in that order) for each term $i^*$ marked in $B$.
Since $B$ is a ballot sequence, it follows that $T$ is an $\Sp_{2k}$-tableau such that $\mathbf{wt}(T) = \bm{\delta}$, using the analogue of~\eqref{i inequality} again.
\end{proof}

\bibliographystyle{amsplain}
\bibliography{references}

\end{document}

%% file: Table_Howe_duality.tex
\begin{center} 
\begin{tblr}{colspec={|Q[m,c]|Q[m,c]|Q[m,c]|Q[m,c]|Q[m,c]|},stretch=1.5}

\hline

$K$ & $X$ & $G'_{\R}$ & $(\g', \k')$ & {\textup{$K$-action on $f \in \C[X]$, where $g \in K$} \\ \textup{$K'$-action on $f \in \C[X]$, where $g' \in K'$}} \\

\hline[2pt]

$\O_k$ & $\M_{k,n}$ & $\Sp(2n, \R)$ & $(\sp_{2n}, \gl_n)$ & $\begin{aligned}[t] (g \cdot f)(A) &= f(g^{-1}A) \\ (g' \cdot f)(A) &= f(Ag')\end{aligned}$\\

\hline

$\GL_k$ & 
$\M_{k,p} \oplus \M_{k,q}$ & ${\rm U}(p,q)$ & 
$(\gl_{p+q}, \gl_p \oplus \gl_q)$ & $\begin{aligned}[t] (g \cdot f)(A,B) &= f(g^{-1}A, (g^{-1})^T B) \\ ((g'_1, g'_2) \cdot f)(A,B) &= f(Ag'_1, Bg'_2) \end{aligned}$\\

\hline

$\Sp_{2k}$ & $\M_{2k,n}$ & ${\rm SO}^*(2n)$ & $(\so_{2n}, \gl_n)$ & $\begin{aligned}[t] (g \cdot f)(A) &= f(g^{-1}A) \\ (g' \cdot f)(A) &= f(Ag') \end{aligned} $ \\

\hline

\end{tblr}
\end{center}

%% file: Table_Howe_duality_second.tex
\begin{center} 
\begin{tblr}{colspec={|Q[m,c]|Q[m,c]|Q[m,c]|Q[m,c]|Q[m,c]|Q[m,c]|},stretch=1.5}

\hline

$K$ & $(\g', \k')$ & $\Sigma \subseteq \widehat{K}$ & $U^\la$ & $\widetilde{U}^\la$ \\

\hline[2pt]

$\O_k$ & $(\sp_{2n}, \gl_n)$ & $\ell(\la) \leq n$ & $E^\la_k$ & $\widetilde{E}^\la_{2n}$\\

\hline

$\GL_k$ & 
$(\gl_{p+q}, \gl_p \oplus \gl_q)$ & {$\ell(\la^+) \leq p$, \\ $\ell(\la^-) \leq q $\phantom{,}} & $F^\la_k$ & $\widetilde{F}^\la_{p,q}$\\

\hline

$\Sp_{2k}$ & $(\so_{2n}, \gl_n)$ & $\ell(\la) \leq n$ & $V^\la_{2k}$ & $\widetilde{V}^\la_{2n}$\\

\hline

\end{tblr}
\end{center}

%% file: Table_gkp.tex
\begin{center} 
\begin{tblr}{colspec={|Q[m,c]|Q[m,c]|Q[m,c]|Q[m,c]|Q[m,c]|Q[m,c]|}}

\hline

$K$ & $\g'$ & {$\g'$ \\[1ex] $\left\{\left[\begin{smallmatrix}A&B\\C&D\end{smallmatrix}\right]  : \ldots \right\}$} & {$\k'$ \\[1ex] $\left\{\left[\begin{smallmatrix}A&0\\0&D\end{smallmatrix}\right]\right\}$} & {$\p'_+$ \\[1ex] $\left\{\left[\begin{smallmatrix}
    0 & B \\ 0 & 0
\end{smallmatrix}\right]\right\}$} & $\text{$\C[\p'_+]$ as $\k'$-module}$\\

\hline[2pt]

$\O_k$ & $\sp_{2n}$ & $\begin{aligned}[t] A&=-D^T,\\[-1ex] B & = B^T, \\[-1ex] C &= C^T \end{aligned}$ & $\gl_n$ & ${\rm SM}_n$ & $\quad\displaystyle\bigoplus_{\mathclap{\substack{\nu:\\ \ell(\nu) \leq n,\\ \textup{even rows} \\ \phantom{'}}}}  F^\nu_n$ \\ 

\hline

$\GL_k$ & $\gl_{p+q}$ & no conditions & $\gl_p \oplus \gl_q$ & $\M_{p,q}$ & $\qquad \displaystyle \bigoplus_{\mathclap{\substack{\nu:\\\ell(\nu) \leq \min\{p,q\}}}} F^\nu_p \otimes F^\nu_q$ \\ 

\hline

$\Sp_{2k}$ & $\so_{2n}$ & $\begin{aligned}[t] A&=-D^T,\\[-1ex] B & = -B^T, \\[-1ex] C &= -C^T \end{aligned}$ & $\gl_n$ & ${\rm AM}_n$ & $\quad \displaystyle \bigoplus_{\mathclap{\substack{\nu:\\\ell(\nu) \leq n,\\\textup{even columns}}}} F^\nu_n$ \\

\hline

\end{tblr}
\end{center}

%% file: Table_stable.tex
\begin{center} 
\begin{tblr}{colspec={|Q[m,l]|Q[m,l]|Q[m,l]|Q[m,l]|}, stretch=1.5}

\hline

$K$ & $(\g', \k')$ & \textup{Stable range} &  $\widetilde{U}^\lambda \textup{ as $\k'$-module}$ \\

\hline[2pt]
        
$\O_k$ & $(\sp_{2n}, \gl_n)$ & $k \geq 2n - 1$ & $\widetilde{E}^\la_{2n} \cong \C[{\rm SM}_n] \otimes F^\la_n$ \\ 

\hline

$\GL_k$ & $(\gl_{p+q}, \gl_p \oplus \gl_q)$ & $k \geq p + q - 1$ & $\widetilde{F}^\la_{p,q} \cong \C[\M_{p,q}] \otimes F^{\la^+}_p \otimes F^{\la^-}_q$ \\
        
\hline

$\Sp_{2k}$ & $(\so_{2n}, \gl_n)$ & $k \geq n - 1$ & $\widetilde{V}^{\la}_{2n} \cong \C[{\rm AM}_n] \otimes F^\la_n$ \\

\hline

\end{tblr}
\end{center}

%% file: Table_k1.tex
\begin{center} 
\begin{tblr}{colspec={|Q[m,l]|Q[m,l]|Q[m,l]|Q[m,l]|}, stretch=1.5}

\hline

$K$ & $(\g', \k')$ & $\delta \in \widehat{K}$ & $\widetilde{U}^\delta \textup{ as $\k'$-module}$ \\

\hline[2pt]

$\O_1$ & $(\sp_{2n}, \gl_n)$ & $\delta \in \{0,1\}$ & $\displaystyle\widetilde{E}^{\delta}_{2n} \cong \bigoplus_{\mathclap{\substack{m \in \mathbb{N}:\\ m \equiv \delta \: {\rm mod} \: 2}}} F^{(m)}_n$   \\ 

\hline

$\GL_1$ & $(\gl_{p+q}, \gl_p \oplus \gl_q)$ & $\delta \in \mathbb{Z}$ & $\displaystyle\widetilde{F}^\delta_{p,q} \cong \bigoplus_{\mathclap{\substack{m \in \mathbb{N}: \\ m + \delta \geq 0}}} F^{(m + \delta)}_p \otimes F^{(m)}_q$ \\

\hline

$\Sp_{2}$ & $(\so_{2n}, \gl_n)$ & $\delta \in \mathbb{N}$ & $\displaystyle\widetilde{V}^\delta_{2n} \cong \bigoplus_{m \in \mathbb{N}} F^{(m+\delta,m)}_n$ \\

\hline

\end{tblr}
\end{center}

%% file: main.bbl
\begin{bibdiv}
\begin{biblist}

\bib{BumpSchilling}{book}{
      author={Bump, D.},
      author={Schilling, A.},
       title={Crystal bases: representations and combinatorics},
   publisher={World Scientific Publishing Co. Pte. Ltd., Hackensack, NJ},
        date={2017},
        ISBN={978-981-4733-44-1},
         url={https://doi.org/10.1142/9876},
      review={\MR{3642318}},
}

\bib{Carlitz}{article}{
      author={Carlitz, L.},
       title={Sequences, paths, ballot numbers},
        date={1972},
        ISSN={0015-0517},
     journal={Fibonacci Quart.},
      volume={10},
      number={5},
       pages={531\ndash 549},
      review={\MR{317949}},
}

\bib{Casselman}{inproceedings}{
      author={Casselman, W.},
       title={Jacquet modules for real reductive groups},
        date={1980},
   booktitle={Proceedings of the {I}nternational {C}ongress of {M}athematicians ({H}elsinki, 1978)},
   publisher={Acad. Sci. Fennica, Helsinki},
       pages={557\ndash 563},
      review={\MR{562655}},
}

\bib{CEW}{article}{
      author={Colarusso, M.},
      author={Erickson, W.},
      author={Willenbring, J.},
       title={Contingency tables and the generalized {L}ittlewood--{R}ichardson coefficients},
        date={2022},
        ISSN={0002-9939},
     journal={Proc. Amer. Math. Soc.},
      volume={150},
      number={1},
       pages={79\ndash 94},
         url={https://doi.org/10.1090/proc/15731},
      review={\MR{4335859}},
}

\bib{Collingwood}{article}{
      author={Collingwood, D.},
       title={Harish-{C}handra modules with the unique embedding property},
        date={1984},
        ISSN={0002-9947,1088-6850},
     journal={Trans. Amer. Math. Soc.},
      volume={281},
      number={1},
       pages={1\ndash 48},
         url={https://doi.org/10.2307/1999521},
      review={\MR{719657}},
}

\bib{EHP}{incollection}{
      author={Enright, T.},
      author={Hunziker, M.},
      author={Pruett, W.~A.},
       title={Diagrams of {H}ermitian type, highest weight modules, and syzygies of determinantal varieties},
        date={2014},
   booktitle={Symmetry: representation theory and its applications},
      series={Progr. Math.},
      volume={257},
   publisher={Birkh\"{a}user/Springer, New York},
       pages={121\ndash 184},
         url={https://doi.org/10.1007/978-1-4939-1590-3_6},
      review={\MR{3363009}},
}

\bib{EHW}{article}{
      author={Enright, T.},
      author={Hunziker, M.},
      author={Wallach, N.},
       title={A {P}ieri rule for {H}ermitian symmetric pairs. {I}},
        date={2004},
        ISSN={0030-8730},
     journal={Pacific J. Math.},
      volume={214},
      number={1},
       pages={23\ndash 30},
         url={https://doi.org/10.2140/pjm.2004.214.23},
      review={\MR{2039124}},
}

\bib{EW}{article}{
      author={Enright, T.},
      author={Willenbring, J.},
       title={Hilbert series, {H}owe duality and branching for classical groups},
        date={2004},
        ISSN={0003-486X},
     journal={Ann. of Math. (2)},
      volume={159},
      number={1},
       pages={337\ndash 375},
         url={https://doi.org/10.4007/annals.2004.159.337},
      review={\MR{2052357}},
}

\bib{Frohmader}{unpublished}{
      author={Frohmader, A.},
       title={Graded multiplicities in the {K}ostant--{R}allis setting},
        date={2023},
        note={arXiv:2312.11295},
}

\bib{Fulton}{book}{
      author={Fulton, W.},
       title={Young tableaux},
      series={London Mathematical Society Student Texts},
   publisher={Cambridge University Press, Cambridge},
        date={1997},
      volume={35},
        ISBN={0-521-56144-2; 0-521-56724-6},
      review={\MR{1464693}},
}

\bib{FH}{book}{
      author={Fulton, W.},
      author={Harris, J.},
       title={Representation theory: a first course},
      series={Graduate Texts in Mathematics},
   publisher={Springer-Verlag, New York},
        date={1991},
      volume={129},
        ISBN={0-387-97527-6; 0-387-97495-4},
         url={https://doi.org/10.1007/978-1-4612-0979-9},
      review={\MR{1153249}},
}

\bib{GW}{book}{
      author={Goodman, R.},
      author={Wallach, N.},
       title={Symmetry, {R}epresentations, and {I}nvariants},
      series={Graduate Texts in Mathematics},
   publisher={Springer, Dordrecht},
        date={2009},
      volume={255},
        ISBN={978-0-387-79851-6},
         url={https://doi.org/10.1007/978-0-387-79852-3},
      review={\MR{2522486}},
}

\bib{HC55}{article}{
      author={Harish-Chandra},
       title={Representations of semisimple {L}ie groups. {IV}},
        date={1955},
     journal={Amer. J. Math.},
      volume={77},
       pages={743\ndash 777},
}

\bib{HC56}{article}{
      author={Harish-Chandra},
       title={Representations of semisimple {L}ie groups. {V}},
        date={1956},
     journal={Amer. J. Math.},
      volume={78},
       pages={1\ndash 41},
}

\bib{Howe89}{article}{
      author={Howe, R.},
       title={Remarks on classical invariant theory},
        date={1989},
        ISSN={0002-9947},
     journal={Trans. Amer. Math. Soc.},
      volume={313},
      number={2},
       pages={539\ndash 570},
         url={https://doi.org/10.2307/2001418},
      review={\MR{986027}},
}

\bib{HK}{incollection}{
      author={Howe, R.},
      author={Kraft, H.},
       title={Principal covariants, multiplicity-free actions, and the {$K$}-types of holomorphic discrete series},
        date={1998},
   booktitle={Geometry and representation theory of real and {$p$}-adic groups ({C}\'{o}rdoba, 1995)},
      series={Progr. Math.},
      volume={158},
   publisher={Birkh\"{a}user},
     address={Boston, MA},
       pages={147\ndash 161},
      review={\MR{1486139}},
}

\bib{HoweLee}{article}{
      author={Howe, R.},
      author={Lee, S.T.},
       title={Why should the {L}ittlewood-{R}ichardson rule be true?},
        date={2012},
        ISSN={0273-0979,1088-9485},
     journal={Bull. Amer. Math. Soc. (N.S.)},
      volume={49},
      number={2},
       pages={187\ndash 236},
         url={https://doi.org/10.1090/S0273-0979-2011-01358-1},
      review={\MR{2888167}},
}

\bib{HTW}{article}{
      author={Howe, R.},
      author={Tan, E.-C.},
      author={Willenbring, J.},
       title={Stable branching rules for classical symmetric pairs},
        date={2005},
     journal={Trans. Amer. Math. Soc.},
      volume={357},
      number={4},
       pages={1601\ndash 1626},
      review={\MR{2115378} (2005j:22007)},
}

\bib{KV}{article}{
      author={Kashiwara, M.},
      author={Vergne, M.},
       title={On the {S}egal--{S}hale--{W}eil representations and harmonic polynomials},
        date={1978},
        ISSN={0020-9910},
     journal={Invent. Math.},
      volume={44},
      number={1},
       pages={1\ndash 47},
         url={https://doi.org/10.1007/BF01389900},
      review={\MR{463359}},
}

\bib{KY}{article}{
      author={Kim, S.},
      author={Yacobi, O.},
       title={A basis for the symplectic group branching algebra},
        date={2012},
        ISSN={0925-9899,1572-9192},
     journal={J. Algebraic Combin.},
      volume={35},
      number={2},
       pages={269\ndash 290},
         url={https://doi.org/10.1007/s10801-011-0303-7},
      review={\MR{2886291}},
}

\bib{KingGYT}{article}{
      author={King, R.},
       title={Generalized {Y}oung tableaux and the general linear group},
        date={1970},
        ISSN={0022-2488,1089-7658},
     journal={J. Mathematical Phys.},
      volume={11},
       pages={280\ndash 293},
         url={https://doi.org/10.1063/1.1665059},
      review={\MR{251972}},
}

\bib{King75}{article}{
      author={King, R.},
       title={Branching rules for classical {L}ie groups using tensor and spinor methods},
        date={1975},
        ISSN={0305-4470,1751-8121},
     journal={J. Phys. A},
      volume={8},
       pages={429\ndash 449},
         url={http://stacks.iop.org/0305-4470/8/429},
      review={\MR{411400}},
}

\bib{KingEl}{article}{
      author={King, R.},
      author={El-Sharkaway, N.},
       title={Standard {Y}oung tableaux and weight multiplicities of the classical {L}ie groups},
        date={1983},
        ISSN={0305-4470,1751-8121},
     journal={J. Phys. A},
      volume={16},
      number={14},
       pages={3153\ndash 3177},
         url={http://stacks.iop.org/0305-4470/16/3153},
      review={\MR{725604}},
}

\bib{KingWelshOrthogonal}{article}{
      author={King, R.},
      author={Welsh, T.},
       title={Construction of orthogonal group modules using tableaux},
        date={1993},
        ISSN={0308-1087,1563-5139},
     journal={Linear and Multilinear Algebra},
      volume={33},
      number={3-4},
       pages={251\ndash 283},
         url={https://doi.org/10.1080/03081089308818198},
      review={\MR{1334676}},
}

\bib{KT}{article}{
      author={Koike, K.},
      author={Terada, I.},
       title={Young diagrammatic methods for the restriction of representations of complex classical {L}ie groups to reductive subgroups of maximal rank},
        date={1990},
        ISSN={0001-8708,1090-2082},
     journal={Adv. Math.},
      volume={79},
      number={1},
       pages={104\ndash 135},
         url={https://doi.org/10.1016/0001-8708(90)90059-V},
      review={\MR{1031827}},
}

\bib{Kudla}{incollection}{
      author={Kudla, S.},
       title={Seesaw dual reductive pairs},
        date={1984},
   booktitle={Automorphic forms of several variables ({K}atata, 1983)},
      series={Progr. Math.},
      volume={46},
   publisher={Birkh\"auser Boston, Boston, MA},
       pages={244\ndash 268},
      review={\MR{763017}},
}

\bib{Lepowsky}{article}{
      author={Lepowsky, J.},
       title={Multiplicity formulas for certain semisimple {L}ie groups},
        date={1971},
        ISSN={0002-9904},
     journal={Bull. Amer. Math. Soc.},
      volume={77},
       pages={601\ndash 605},
         url={https://doi.org/10.1090/S0002-9904-1971-12767-2},
      review={\MR{301142}},
}

\bib{Littelmann94}{article}{
      author={Littelmann, P.},
       title={A {L}ittlewood-{R}ichardson rule for symmetrizable {K}ac-{M}oody algebras},
        date={1994},
        ISSN={0020-9910,1432-1297},
     journal={Invent. Math.},
      volume={116},
      number={1-3},
       pages={329\ndash 346},
         url={https://doi.org/10.1007/BF01231564},
      review={\MR{1253196}},
}

\bib{Littelmann95}{article}{
      author={Littelmann, P.},
       title={Paths and root operators in representation theory},
        date={1995},
        ISSN={0003-486X,1939-8980},
     journal={Ann. of Math. (2)},
      volume={142},
      number={3},
       pages={499\ndash 525},
         url={https://doi.org/10.2307/2118553},
      review={\MR{1356780}},
}

\bib{Macdonald}{book}{
      author={Macdonald, I.~G.},
       title={Symmetric functions and {H}all polynomials},
     edition={2},
      series={Oxford Mathematical Monographs},
   publisher={Oxford University Press},
     address={New York},
        date={1995},
        ISBN={0-19-853489-2},
      review={\MR{1354144}},
}

\bib{ProctorRSK}{article}{
      author={Proctor, R.},
       title={A {S}chensted algorithm which models tensor representations of the orthogonal group},
        date={1990},
        ISSN={0008-414X,1496-4279},
     journal={Canad. J. Math.},
      volume={42},
      number={1},
       pages={28\ndash 49},
         url={https://doi.org/10.4153/CJM-1990-002-1},
      review={\MR{1043509}},
}

\bib{Proctor}{article}{
      author={Proctor, R.},
       title={Young tableaux, {G}elfand patterns, and branching rules for classical groups},
        date={1994},
        ISSN={0021-8693,1090-266X},
     journal={J. Algebra},
      volume={164},
      number={2},
       pages={299\ndash 360},
         url={https://doi.org/10.1006/jabr.1994.1064},
      review={\MR{1271242}},
}

\bib{Schmid}{article}{
      author={Schmid, W.},
       title={Die {R}andwerte holomorpher {F}unktionen auf hermitesch symmetrischen {R}\"aumen},
        date={1969/70},
        ISSN={0020-9910,1432-1297},
     journal={Invent. Math.},
      volume={9},
       pages={61\ndash 80},
         url={https://doi.org/10.1007/BF01389889},
      review={\MR{259164}},
}

\bib{SchmidNotes}{incollection}{
      author={Schmid, W.},
       title={Geometric methods in representation theory},
        date={2005},
   booktitle={Poisson geometry, deformation quantisation and group representations},
      series={London Math. Soc. Lecture Note Ser.},
      volume={323},
   publisher={Cambridge Univ. Press, Cambridge},
       pages={273\ndash 323},
        note={Lecture notes taken by Matvei Libine},
      review={\MR{2166454}},
}

\bib{Schwarz}{article}{
      author={Schwarz, G.},
       title={Representations of simple {L}ie groups with a free module of covariants},
        date={1978},
        ISSN={0020-9910,1432-1297},
     journal={Invent. Math.},
      volume={50},
      number={1},
       pages={1\ndash 12},
         url={https://doi.org/10.1007/BF01406465},
      review={\MR{516601}},
}

\bib{Stanley}{book}{
      author={Stanley, R.},
       title={Enumerative combinatorics. {V}olume 1},
     edition={Second},
      series={Cambridge Studies in Advanced Mathematics},
   publisher={Cambridge University Press, Cambridge},
        date={2012},
      volume={49},
        ISBN={978-1-107-60262-5},
      review={\MR{2868112}},
}

\bib{Stembridge}{article}{
      author={Stembridge, J.},
       title={Rational tableaux and the tensor algebra of {$\mathfrak{gl}_n$}},
        date={1987},
        ISSN={0097-3165,1096-0899},
     journal={J. Combin. Theory Ser. A},
      volume={46},
      number={1},
       pages={79\ndash 120},
         url={https://doi.org/10.1016/0097-3165(87)90077-X},
      review={\MR{899903}},
}

\bib{Sundaram}{article}{
      author={Sundaram, S.},
       title={Orthogonal tableaux and an insertion algorithm for {${\rm SO}(2n+1)$}},
        date={1990},
        ISSN={0097-3165,1096-0899},
     journal={J. Combin. Theory Ser. A},
      volume={53},
      number={2},
       pages={239\ndash 256},
         url={https://doi.org/10.1016/0097-3165(90)90059-6},
      review={\MR{1041447}},
}

\bib{WallachRRI}{book}{
      author={Wallach, N.},
       title={Real reductive groups. {I}},
      series={Pure and Applied Mathematics},
   publisher={Academic Press, Inc., Boston, MA},
        date={1988},
      volume={132},
        ISBN={0-12-732960-9},
      review={\MR{929683}},
}

\bib{Yacobi}{article}{
      author={Yacobi, O.},
       title={An analysis of the multiplicity spaces in branching of symplectic groups},
        date={2010},
        ISSN={1022-1824,1420-9020},
     journal={Selecta Math. (N.S.)},
      volume={16},
      number={4},
       pages={819\ndash 855},
         url={https://doi.org/10.1007/s00029-010-0033-z},
      review={\MR{2734332}},
}

\end{biblist}
\end{bibdiv}
